\documentclass{article}

    \usepackage[preprint]{neurips_2020}

\usepackage[utf8]{inputenc} 
\usepackage[T1]{fontenc}    
\usepackage{url}            
\usepackage{booktabs}       
\usepackage{amsfonts}       
\usepackage{nicefrac}       
\usepackage{microtype}      
\usepackage{algorithm}
\usepackage{algorithmic}

\usepackage{graphicx}
\usepackage{caption}
\usepackage{amsfonts}
\usepackage{amsmath}
\usepackage{amssymb}
\usepackage{fancyhdr}
\usepackage{titlesec}
\usepackage{indentfirst}
\usepackage{booktabs}
\usepackage{verbatim}
\usepackage{color}
\usepackage{tcolorbox}
\usepackage{mathtools}
\usepackage{bm}
\usepackage{amsthm}
\usepackage{wasysym}
\usepackage{empheq}

\newtheorem{theorem}{Theorem}[section]
\newtheorem{remark}[theorem]{Remark}

\newtheorem{lemma}[theorem]{Lemma}
\newtheorem{corollary}[theorem]{Corollary}

\newcommand{\argmin}{\operatornamewithlimits{argmin}}
\newcommand{\argmax}{\operatornamewithlimits{argmax}}

\usepackage{tikz}
\newcommand*\circled[1]{\tikz[baseline=(char.base)]{
            \node[shape=circle,draw,inner sep=0.8pt] (char) {#1};}}
\def\a{\alpha}

\def\b{\beta}
\def\d{\delta}
\def\e{\epsilon}
\def\g{\gamma}
\def\lam{\lambda}
\def\o{\omega}

\def\s{\sigma}
\def\th{\theta}
\def\v{\varepsilon}
\def\th{\theta}

\def\O{\Omega}
\def\D{\Delta}

\def\A{\mathbb A}

\def\E{\mathbb E}

\def\R{\mathbb R}
\def\S{\mathbb S}
\def\P{\mathbb P}
\def\V{\mathbb V}

\def\mP{\mathcal P}

\def\l{\left}
\def\r{\right}

\def\ll{\left\lVert}
\def\rl{\right\rVert}
\def\lv{\left\lvert}
\def\rv{\right\rvert}
\def\({\left(}
\def\){\right)}

\def\dsp{\displaystyle}

\def\pt{\partial}
\def\nb{\nabla}

\def\ds{\displaystyle}

\def\qd{\quad}

\def\h{\hat}
\def\t{\tilde}

\def\ds{{d_s}}
\def\da{{d_a}}

\def\sm{{s_m}}
\def\zm{Z_m}
\def\dsm{\Delta s_m}
\def\am{{a_m}}

\def\smp{s_{m+1}}
\def\zmp{Z_{m+1}}
\def\dsmp{\Delta s_{m+1}}
\def\smpt{s_{m+2}}
\def\amp{{a_{m+1}}}

\def\Q{Q^\pi}

\def\se{\sqrt{\e}}
\def\dth{{d_\th}}

\def\hF{\h{F}}
\def\tF{\t{F}}
\def\tJ{\t{J}}

\def\hp{\h{p}}
\def\tp{\t{p}}
\def\hd{\h{d}}
\def\hj{\h{j}}
\def\td{\t{d}}
\def\st{*}
\def\Qst{Q^{\st}}
\def\tth{\t{\th}}
\def\pinf{p^\infty}

\def\hf{\h{f}}

\def\dthst{d_{\th_\st}}
\def\T{\mathbb{T}}

\def\Ds{\D s}
\def\je{j^{\textrm{eval}}}
\def\jc{j^{\textrm{ctrl}}}
\def\Th{\Theta}

\title{Borrowing From the Future: Addressing Double Sampling in Model-free Control}

\author{Yuhua Zhu \\
  Department of Mathematics\\
  Stanford University\\
  \texttt{yuhuazhu@stanford.edu} \\
   \And
   Zach Izzo \\
    Department of Mathematics\\
  Stanford University\\
   \texttt{zizzo@stanford.edu} \\
   \And
   Lexing Ying\\
    Department of Mathematics\\
   and\\
      Institute for Computational and Mathematical Engineering\\
  Stanford University\\
   \texttt{lexing@stanford.edu} \\
}

\begin{document}

\maketitle

\begin{abstract}
In model-free reinforcement learning, the temporal difference method and its variants become
unstable when combined with nonlinear function approximations. Bellman residual minimization with
stochastic gradient descent (SGD) is more stable, but it suffers from the double sampling problem:
given the current state, two independent samples for the next state are required, but often only one sample is available. Recently, the authors of \cite{zhu2020} introduced the 
borrowing from the future (BFF) algorithm to address this issue for the prediction problem. The main
idea is to borrow extra randomness from the future to approximately re-sample the next state when the underlying dynamics of the problem are sufficiently smooth. This paper extends the BFF algorithm to action-value
function based model-free control. We prove that BFF is close to unbiased SGD when the underlying
dynamics vary slowly with respect to actions. We confirm our theoretical findings with
numerical simulations.
\end{abstract}


  

\section{Introduction}
\label{sec:Intro}
\paragraph{Background}
The goal of reinforcement learning (RL) is to find an optimal policy which maximizes the return of a
Markov decision process (MDP) \cite{sutton2018reinforcement}.  One of the most common ways of
finding an optimal policy is to treat it as the fixed point of the Bellman operator.
Researchers have developed efficient iterative methods such as temporal difference (TD) \cite{sutton1988learning}, $Q$-learning
\cite{Watkin1989}, and SARSA \cite{Rummery1994} based on the contraction property of the Bellman operator.

Nonlinear function approximations have recently received a great deal of attention in RL. This follows
the successful application of neural networks (NNs) to Atari games \cite{mnih2013playing,Mnih2015}, as well as in
Alpha Go and Alpha Zero \cite{Silver2016, Silver2017}. However, when using a nonlinear approximation
and off-policy data, the Bellman operator fails to retain the contraction
property. The result is that training naive NN approximation may be unstable. Many variants and modifications
have been proposed to stabilize training. For example, DQN\cite{Mnih2015} and A3C
\cite{Mnih2016} stabilize $Q$-learning by using a slowly changing target network and replaying over
past experiences or using parallel agents for exploration; double DQN reduces instability by
using two separate $Q$ value estimators, one for choosing the action and the other for evaluating
the action's quality \cite{hasselt2015}.


Another way to stabilize RL with a nonlinear approximation is to formulate it as a minimization
problem. This approach is known as Bellman residual minimization (BRM) \cite{baird1995}. However,
applying stochastic gradient descent (SGD) to BRM directly suffers from the so-called double
sampling problem: at a given state, two independent samples for the next state are required in order
to perform unbiased SGD. Such a requirement is often hard to fulfill in a model-free setting,
especially for problems with a continuous state space.

\paragraph{Contributions}
In this paper, we revisit BRM for $Q$-value prediction and control problems in the model-free RL
setting. The main assumption is that the underlying dynamics of the MDP can be
written as $\E[\smp - \sm | \sm, \am] = \mu(\sm, \am)\e$, where $\e$ is a small parameter. Note that
knowledge of the dynamics is not required to implement the algorithm.  We extend the
borrowing-from-the-future (BFF) algorithm of \cite{zhu2020} to action-value based RL. The key idea
is to borrow extra randomness from the future by leveraging the smoothness of the
underlying RL problem. We prove that when the underlying dynamics change slowly with respect to
actions and the policy changes slowly with respect to states, the training trajectory of the
proposed algorithm is statistically close to the training trajectory of unbiased SGD. The difference
between the two algorithms will first decay exponentially and eventually stabilize at an error of
$O(\e\d_\st)$, where $\d_\st$ is the smallest Bellman residual that unbiased SGD can achieve.

\section{Models and key ideas}

\subsection{Continuous state space}
\label{sec: setting}
In model-free RL, consider a discrete-time MDP with continuous state space $\S \subset\R^{\ds}$. The
action space $\A \subset \R^{\da}$ maybe be continuous or discrete. We denote the transition kernel of the MDP as
\begin{equation}
  \label{trans matrix}
  P^a(s,s') = \P\l(\smp = s'| \sm = s,\am = a\r).
\end{equation}
The immediate reward function $r(s',s, a)$ specifies the reward if one takes action $a$ at state $s$
and ends up at state $s'$. A policy $\pi(a|s)$ gives the probability of taking action $a$ at state
$s$, i.e., $\P\l\{\text{take action }a \text{ at state }s\r\} = \pi(a \vert s).$ For a continuous
state space, it is often convenient to rewrite the underlying transition in terms of the states:
\begin{equation}  \label{def of transition}
  \smp = \sm + \mu(\sm, a)\e + \se \zm,
\end{equation}
where $\zm$ is a mean-zero noise term. This form is particularly relevant when the MDP arises as 
a discretization of an underlying stochastic differential equation (SDE), with $\e$ as its
discretized time step. We remark that this SDE interpretation is not necessary; our theorems and
algorithms apply to more general MDPs as long as the difference between the current and next state
can be written as
\begin{equation}
\label{eq: trans exp}
    \E[\smp - \sm | \sm, \am] = \mu(\sm, \am)\e.
\end{equation}
Throughout the paper, we consider the case where for each state, the variation of the underlying
drift $\mu(s,a)$ is a priori bounded in the action space.

The main object under study is the action-state pair value function $Q(s,a)$. There are
two types of problems: $Q$-evaluation and $Q$-control. $Q$-evaluation refers to the prediction
of the value function when the policy is given, while $Q$-control refers to finding the optimal policy through the
maximization of $Q(s,a)$. For the $Q$-evaluation problem the state space and action space can be
continuous or discrete, while for the $Q$-control problem we mainly consider the case of a (finite) discrete
action space.


\paragraph{$Q$-evaluation}
Given a fixed policy $\pi$, the value function $\Q(s,a)$ represents the expected return if one takes
action $a$ at state $s$ and follows $\pi$ thereafter, i.e.,
\begin{equation*}
  \Q(s, a) = \E \l[\l. \sum_{t\geq0} \g^t r(s_{m+t+1},s_{m+t},a_{m+t}) \r\vert \sm = s, \am = a \r],
\end{equation*}
where $\g\in(0,1)$ is a discount factor. The value function $\Q$ satisfies the Bellman equation
\cite{sutton2018reinforcement} $\Q(s, a) = \T^{\pi} \Q(s,a)$, where
\begin{equation}
\label{eq: bellman}
     \qd \T^\pi \Q(s,a) = \E\l[\l.r(\smp,\sm,\am)+\g\Q(\smp, \amp) \r\vert(\sm, \am) = (s, a) \r].
\end{equation}
In the nonlinear approximation setting, one seeks a solution to
\eqref{eq: bellman} from a family of functions $\Q(s, a;\th)$ parameterized by $\th\in\R^\dth$. For example, the function approximation family could be the set of all NNs of a given architecture, and $\th$ specifies the network weights. One way to find
$\Q(s,a;\th)$ is to solve the following {\it Bellman residual minimization (BRM) problem}:
\begin{equation}
  \label{eq: optimization}
  \min_{\th\in\R^\dth} \underset{(s,a)\sim\rho(s,a)}{\E} \d^2(s,a;\th)
\end{equation}
where $\rho(s,a)$ is a distribution over $\S\times \A$ and 
\begin{equation}
  \label{def of delta}
  \d(s,a;\th) = \T^\pi \Q(s,a;\th) - \Q(s,a;\th).
\end{equation}
Note that the expectation in \eqref{eq: optimization} can be taken with respect to different
distributions $\rho$. For online learning, it is often the stationary distribution of the Markov
chain. When $\S$ and $\A$ are discrete, it is also reasonable to choose a uniform distribution over
$\S\times\A$. Doing so often accelerates the rate of convergence compared to the stationary measure.

One approach for solving the Bellman minimization problem \eqref{eq: optimization} is to
directly apply SGD. The unbiased gradient estimate to the loss function is
\begin{equation}
  \label{def of F eval}
  F= j(\sm, \am, \smp;\th_m) \nb_\th j(\sm, \am, \smp'; \th_m),
\end{equation}
where 
\begin{equation}
  \label{def of j}
  j(\sm, \am, \smp;\th_m) = r(\smp,\sm,\am)+\g\int \Q(\smp,a;\th)\pi(a\vert\smp)da\  - \Q(\sm, \am;\th).
\end{equation}
Here $\smp$ is the next state in the trajectory, while $\smp'$ is an independent sample for the next
state according to the transition process. However, in model-free RL, as the underlying dynamics are
unknown, another independent sample $\smp'$ of the next state is unavailable. Therefore, this
unbiased SGD, refered to as {\em uncorrelated sampling} (US), is impractical. Even if one can store the
whole trajectory, it is impossible to revisit a certain state multiple times when the state space is
either continuous or discrete but of high dimension. This is the so-called {\em double sampling
  problem}. One potential solution, called {\em sample-cloning} (SC), simply uses $\smp$ as a surrogate for
$\smp'$, i.e. $s'_{m+1} = s_{m+1}$.  However, sample-cloning is not an unbiased algorithm for the
BRM problem, and its bias grows rapidly with the conditional variance of $\smp$ on $\sm$.



To address the double sampling problem, \cite{zhu2020} introduced the borrowing from the future
(BFF) algorithm. The main idea of the BFF algorithm is to borrow the future difference
$\dsmp=\smpt-\smp$ and approximate the second sample $\smp'$ with $\sm + \dsmp$. During SGD, the
parameter $\th$ is updated based on the following estimate of the unbiased gradient:
\begin{equation}  \label{def of hF eval}
  \hF = j(\sm, \am, \smp;\th_m) \nb_\th j(\sm, \am, \sm+\dsmp;\th_m),
\end{equation}
where $j$ is defined in \eqref{def of j}.  
When the difference between $\dsm$ and $\dsmp$ is small, the new $\smp'$ is statistically close to
the distribution of the true next state. Among the two versions (gradient based and loss function
based) introduced in \cite{zhu2020}, we adopt the gradient version, detailed in Algorithm \ref{algo:
  bff para}. In Section \ref{sec: related}, we comment on why the loss version is less accurate.

\begin{algorithm}
  \caption{BFF}
  \label{algo: bff para}
  \begin{algorithmic}[1]
    \REQUIRE $\eta$: Learning rate
    \REQUIRE $Q^\pi(s; \th) \in \R^{|\A|}$ or $\Q(s,a;\th)\in\R$: Nonlinear approximation of $Q$ parameterized by $\th$
    \REQUIRE $\je(s, a, s'; \th) := r(s', s, a) + \g \int Q^\pi(s', a; \th)\pi(a|s')da - Q^\pi(s, a; \th)$
    \REQUIRE $\th_0$: Initial parameter vector
    \STATE $m \gets 0$
    \WHILE{$\th_m$ not converged}
        \STATE $s'_{m+1} \gets s_m + (s_{m+2} - s_{m+1})$
        \STATE $\hat{F}_m \gets \je(s_m, a_m, s_{m+1}; \th_m)\nabla_\th \je(s_m, a_m, s'_{m+1}; \th_m)$
        \STATE $\th_{m+1} \gets \th_m - \eta \hat{F}_m$
        \STATE $m \gets m+1$
    \ENDWHILE
\end{algorithmic}
\end{algorithm}

Due to the Markov property, the difference $\dsmp$ is independent from the current difference
$\dsm$, leading to two conditionally independent samples. Whether $\hF$ is a good approximation of the unbiased
estimate $F$ depends on three factors: 1) the variation of the drift $\mu(s,a)$ over the action
space; 2) the variation of the policy $\pi(a|s)$ over the state space; 3) the size of $\e$. The
smaller these three elements are, the closer BFF is to US.

In Algorithm \ref{algo: bff para}, only one future step is used for generating a new sample of
$\smp$. In order to reduce the variance of the BFF gradient, it is useful to consider
replacing the future step by a weighted average of multiple future steps. The estimate of the
gradient then takes the form
\begin{equation} \label{eq: nbff}
  \hF^n = j(\sm, \am, \smp;\th_m) \sum_{i = 1}^n \a_i\nb_\th j(\sm, \am, \sm+\Ds_{m+i}; \th_m)
\end{equation}
with $\sum_i {\a_i} = 1$. This comes at the cost of potentially increasing the estimate's
bias.
 
\paragraph{$Q$-control}
The BFF algorithm mentioned above can be extended easily to $Q$-control, i.e., finding the value
function $\Qst$ of the optimal policy $\pi_\st$. $\Qst$ satisfies the Bellman equation $\Qst(s, a) =
\T^{\pi_\st}\Qst(s,a)$, where
\begin{equation}
\label{eq: bellman 2}
\T^{\pi_\st}\Qst(s,a)=\E\l[\l.r(\smp,\sm,\am)+\g\max_{a'}\Qst(\smp,a';\th) \r\vert(\sm,\am) =
  (s, a) \r].
\end{equation}
The BRM problem is the same as \eqref{eq: optimization} but with the Bellman residual $\d(s,a;\th)$
given by $\delta(s,a;\th)= \T^{\pi_\st}\Qst(s,a) - \Qst(s,a)$. Rather than generating a trajectory offline with a fixed policy, we instead generate a training trajectory
online using an $\e$-greedy policy. The algorithm for this case is identical to Algorithm \ref{algo:
  bff para}, but with $\je$ replaced by $\jc(s_m, a_m, s_{m+1}; \th) = r(s_{m+1}, s_m, a_m) + \g
\max_a Q^*(s_{m+1}, a; \th) - Q^*(s_m, a_m; \th)$.
Refer to Appendix \ref{appendix: control algo} for more details. 

\paragraph{Why BFF works}
We prove in Lemma \ref{lemma: diff uc-gd sc-gd} and \ref{lemma: control diff uc-gd sc-gd} that the
difference between the SC and US gradients is $O(\e)$, while the difference between the BFF
and US gradients is $O(\E[\d\e])$ (see Lemma \ref{lemma: diff uc-gd BFF short}). Although both differences are
$O(\e)$, BFF depends on the Bellman residual $\d$ while SC does not. As the algorithm proceeds, $\d$
approaches $0$, causing the difference between BFF and US to further decrease. On the other
hand, the difference between SC and unbiased SGD is always $O(\e)$. This is the high-level reason
why BFF outperforms SC. (See Section \ref{sec: numerics} for numerical comparisons.)

\subsection{Discrete state space}


When the state space is discrete, one can view $Q\in\R^{|\S|\times |\A|}$ as a matrix. In this tabular form, one can directly use the previous function approximation framework
by letting $\Q(s,a;\th) = \Phi(s,a)^\top\th $, where $\Phi(s_i,a_j)\in\R^{|\S|\times |\A|}$ is the matrix with $(i,j)$-th entry equal to $1$ and all other entries equal to $0$.
Equivalently, one can also derive the BFF algorithm directly by computing the gradient of the
Bellman residual with respect to $Q$. An unbiased gradient is given by
\begin{align*}
    F(s_m, a_m) &= -\je(s_m, a_m, s_{m+1}) \\
    F(s'_{m+1}, a) &= \pi(a|s'_{m+1})\g \je(s_m, a_m, s_{m+1}), \quad \forall a \in \A,
\end{align*}
where $s'_{m+1}$ is an independent sample of the next step in the trajectory given $s_m$ and $a_m$,
$\je(s_m, a_m, s_{m+1}) = r(s_{m+1}, s_m, a_m) + \g \sum_a Q^\pi(s_{m+1}, a)\pi(a|s) - Q^\pi(s_m,
a_m)$, and all other entries of $F$ are 0. By replacing the independent sample $s'_{m+1}$ with the
BFF approximation $s_m + \dsm$, we obtain the following BFF algorithm for the tabular case, 
summarized in Algorithm \ref{algo: bff tab}.

%
%
%
%
%

\begin{algorithm}
  \caption{BFF (tabular case)}
  \label{algo: bff tab}
  \begin{algorithmic}[1]
    \REQUIRE $\eta$: Learning rate
    \REQUIRE $Q^\pi \in \R^{|\S|\times |\A|}$: matrix of $Q^\pi(s,a)$ values
    \REQUIRE $\je(s_m, a_m, s_{m+1}) = r(s_{m+1}, s_m, a_m) + \g \sum_a Q^\pi(s_{m+1}, a)\pi(a|s) - Q^\pi(s_m, a_m)$
    \STATE $m \gets 0$
    \WHILE{$Q^\pi$ not converged}
    \STATE $s'_{m+1} \gets s_m + (s_{m+2} - s_{m+1})$
    \STATE $\hat{F}_m \gets 0 \in \R^{|\S|\times |\A|}$
    \STATE $\hat{F}_m(s_m, a_m) \gets -\je(s_m, a_m, s_{m+1})$
    \FOR{$a \in \A$}
    \STATE $\hat{F}_m(s'_{m+1}, a) \gets \pi(a|s'_{m+1})\g\je(s_m, a_m, s_{m+1})$
    \ENDFOR
    \STATE $Q^\pi \gets Q^\pi - \eta \hat{F}_m$
    \STATE $m \gets m+1$
    \ENDWHILE
  \end{algorithmic}
\end{algorithm}
The $Q$-control algorithm for the tabular case can be found in Appendix \ref{appendix: control algo}. As in equation (\ref{eq: nbff}), one can use multiple future steps to reduce the variance of the gradient in the tabular case as well. Refer to Appendix \ref{appendix: nbff tabular} for more details.

\subsection{Related work}
\label{sec: related}

There is another version of the BFF algorithm proposed in \cite{zhu2020} for value function
evaluation. One applies the same idea to the loss function instead of the gradient by
minimizing a biased Bellman residual:
\begin{equation}
\begin{aligned}
    \min_{\th\in\R^\dth} \E\l[ \E\l[\l. j(\sm,\am,\smp;\th)j(\sm,\am,\sm+\dsmp;\th)\r\vert\sm, \am
    \r]  \r].
\end{aligned}
\end{equation}
For state value function evaluation, this loss version performs better than the sample-cloning
algorithm because it has a difference of only $O(\e^2)$ from US while SC has an $O(\e)$
difference. However, this loss version does not work for $Q$-evaluation. The reason is that the
gradient of the above loss function contains two parts, $j(\smp) \nb_\th j(\sm+\dsmp) +\nb_\th
j(\smp) j(\sm+\dsmp)$, so the difference between the loss version of BFF and US is $O(\e\d +
\e\nb\d)$, and $\nb\d$ does not necessarily decrease as the algorithm proceeds.
(For example, if $\d^2 = \th^2$, then $\nb\d = 1$ is a constant.) Therefore, the
error is still dominated by an $O(\e)$ term, which means that the loss version behaves similarly to SC.

In \cite{wang2017stochastic,wang2016accelerating}, the stochastic compositional gradient method
(SCGD), a two-step scale algorithm, is proposed to address the double sampling problem.  However, it
is not clear how to apply SCGD to BRM with a continuous state space.


Another way to avoid the double sampling problem in BRM is to consider the primal-dual (PD)
formulation of the minimization problem and view it as a saddle point of a minimax problem. Such
methods include GTD and its variants
\cite{suttongtd2008,Sutton2009,Bhatnagar2018,Mahadevan2011,Liu2015}, and SBEED
\cite{Dai2018}. However, when a nonlinear function approximation is used, the maximum is taken over a non-concave function. This can be significantly more difficult than solving the
minimization problem directly. (See Section \ref{sec: numerics} for details.)



\section{Theoretical results}
\label{sec: main results}

This section states the main theoretical results which bound the difference between BFF and 
US on a continuous state space. 
Recall that the one-step transition is governed by the state dynamics
\begin{equation}  
  \smp = \sm + \mu(\sm, a)\e + \s\se \zm, 
\end{equation}
where $\mu(s,a)$ is the drift, $\zm$ is assumed to be normal $N(0, I_{\ds\times \ds})$, and $\s$ is
the diffusion coefficient. It is convenient to introduce $\dsm := \smp - \sm = \mu(\sm, a)\e + \s\se
\zm$.  For a discrete action space $\A$, the drift term $\{\mu(s, a)\}_{\a\in\A}$ is a family of
continuous functions, while for a continuous action space, $\mu(s,a)$ is a continuous function in
both state and action.  We choose to work with a discretized stochastic differential equation (SDE)
in order to simplify the presentation of the algorithms and the theorems. Our lemmas and theorems
can be extended to the more general case specified by \eqref{eq: trans exp}.


\subsection{Differences at each step}
\label{sec: diff grad}
The following lemma bounds the difference between BFF and US at each
step. That is, assuming the current parameters $\th_m$ are the same, Lemma \ref{lemma: diff uc-gd
  BFF short} bounds the expected difference between BFF and US for $Q$-evaluation and
$Q$-control after one step. See Appendix \ref{appendix: lemma 1} for a more detailed version of Lemma
\ref{lemma: diff uc-gd BFF short} and its proof.
\begin{lemma}[short version]
  \label{lemma: diff uc-gd BFF short}
  For $Q$-evaluation, assume
  \begin{equation}
    \label{ass1: eval}
    \dsp \sup_{s\in\S, \th\in\R^\dth}\lv\pt_s\E_a[ \nb_\th\Q(s, a;\th)|s] \rv\leq C, \qd \dsp
    \sup_{s\in\S, a\in\A}\lv\E_{a'}[\mu(s,a')|s] - \mu(s,a) \rv \leq C, \qd  a.s.
  \end{equation}
  For $Q$-control, let $\dsp f(s;\th) = \max_{a'\in\A} \Qst(s,a;\th)$ and assume 
  \begin{equation}
    \label{ass1: cont}
    \sup_{s\in\S,\th\in\R^\dth}|\pt_s\nb_\th f(\sm;\th)| \leq C , \qd  \sup_{s\in\S,a\in\A}|\E_{a'}[\mu(s,a') | s] -\mu(s,a) |\leq C, \qd  a.s.
  \end{equation}
  The difference between the BFF gradient $\hat{F}$ and the unbiased gradient $F$ is bounded by
  \begin{equation*}
    \lv\E[\hF] - \E[F]\rv\leq \g C^2\E\lv \d (\e  + O(\e))\rv.
  \end{equation*}
\end{lemma}

Note that the upper bounds in the assumptions \eqref{ass1: eval} and \eqref{ass1: cont} that affect
the magnitude of the constant $C$ in front of $\E[\d\e]$ can be translated to assumptions on $\Q$,
$\pi$, and $\mu$. For instance, the first inequality in \eqref{ass1: eval} is satisfied if
$|\pt_s\nb_\th Q|$ and $|\pt_s\pi(a|s)|$ are bounded because $ \lv\pt_s\E_a[ \nb_\th\Q(s, a;\th)|s]
\rv = \lv\pt_s \int \l( \nb_\th\Q(s, a;\th) \pi(a|s) \r) \rv \leq C$. The magnitude of
$|\pt_s\nb_\th Q|$ can be controlled through the function space used to approximate $\Q$. Similarly,
the first equation in \eqref{ass1: cont} is related to $| \pt_s\nb_\th\Qst |$, which can be
controlled through the approximating function space as well.

The second inequality in \eqref{ass1: eval}, \eqref{ass1: cont}, ($\dsp \lv\E_a[\mu(s,a)|s] -
\mu(s,a) \rv \leq C$) is satisfied if $\forall s\in\S$,
\begin{equation}
  \label{cond on mu}
  \begin{aligned}
    &\text{discrete }\A: \qd \max_{a,b\in\A}|\mu(s,a) - \mu(s,b)| \leq C';\\
    &\text{continuous } \A: \qd |\pt_a\mu(s,a)| \leq C'.
  \end{aligned}
\end{equation}

In summary, the crucial elements that affect the difference between BFF and US are 1)
the magnitude of the change in the behavior policy $|\pt_s\pi(a|s)|$ and 2) the variation of the
drift term $\mu(s,a)$ over the action space. Therefore, when the policy changes more slowly with
respect to the state and the drift changes more slowly with respect to the action, the difference is
smaller and BFF performs better.

\subsection{Differences of density evolutions}

This subsection compares the probability density functions (p.d.f.) for the parameters over the
course of the complete BFF and US algorithms. To simplify the analysis, the p.d.f.s of the two
algorithms are modeled with the p.d.f.s of the continuous stochastic processes. The updates of the
parameter $\th_k$ by SGD can be viewed as a discretization of a function in time
$\Th_t\equiv\Th(t)$. 
It is shown in \cite{li2017stochastic,hu2017diffusion} that when the learning rate $\eta$ is small,
the dynamics of SGD can be approximated by a continuous time SDE
\begin{equation}
\label{eq: SDE}
    d\Th_t = - \E [F(\Th_t)] dt + \sqrt{\eta}\V[F(\Th_t)]dB_t
\end{equation}
with $\dsp \Th_{t=k\eta} \approx \th_k$, where $\E$ and $\V$ are expectation and variance taken over 
$\rho(s, a)$. 
Here $\E [F(\Th_t)]$ denotes the true gradient of population loss function in the case of US, or the
biased gradient of the population loss in the case of BFF. For simplicity, we assume $\V[F] \equiv
\xi$ is constant. Let $p(t,\th)$ and $\hp(t,\th)$ be the p.d.f. of the parameter $\th$ at step
$k=t/\eta$ for US and BFF, respectively, and define $\hd(t,\th) = p - \hp$ to be their difference. We
introduce the following weighted norm to measure the difference between the p.d.f.s:
\begin{equation*}
  \ll \hd \rl_\st := \int \hd^2/{\pinf} d\th, \qd \pinf = e^{- \frac{2}{\eta \xi}\E[\d^2]}/Z,
\end{equation*}
where $\pinf$ is the limiting p.d.f. for $p(t,\th)$ as $t\to \infty$, $Z = \int e^{-\b\E[\d^2]}d\th$
is a normalizing constant, and $\d(s,a;\th) = \T^\pi Q - Q$ is the Bellman residual $\T^\pi$ defined
in (\ref{eq: bellman}) for $Q$-evaluation and in \eqref{eq: bellman 2} for $Q$-control. Here the
expectation $\E$ is taken over $\rho(s,a)$. 

\begin{theorem}[short version]
\label{thm: diff of pdf short}
For small $\eta$, the difference $\hd$ of the p.d.f.s for US and BFF is bounded by
\begin{equation}
  \label{eq: diff of pdf short}
  \begin{aligned}
    \ll \hd(t) \rl_\st  \leq& C_1 e^{-C_2t} + O\l(\e\sqrt{\E[\d_\st^2]}\eta^{C_3}\r) \sqrt{1-e^{-C_2t}},
  \end{aligned}
\end{equation}
where $\E[\d^2_\st] = \min_{\th} \E[\d^2]$ and $C_1,C_2,C_3$ are all positive constants.
\end{theorem}

The precise version of Theorem \ref{thm: diff of pdf short} and its proof are given in Appendix
\ref{appendix: thm}. This theorem implies that as the algorithm moves on, the difference between BFF
and US will decay exponentially. After running the algorithm for sufficiently many steps, the difference will
eventually be $O\l(\e\sqrt{\E[\d_\st^2]}\eta^{C_3}\r)$. As long as $\E[\d^2_\st]$ is small, BFF
will achieve a minimizer close to US with an error much smaller than $O(\e)$.  Note that
if $\E[\d^2_\st] = 0$, the difference still does not vanish. Instead, the leading order term of the
last term in \eqref{eq: diff of pdf short} becomes $O(\e\eta^{C_3+1/2})$, which is shown in Corollary
\ref{coro} of Appendix \ref{appendix: thm}.

The constant $C_1$ depends on the initial p.d.f. of the algorithm. The constant $C_3$ is related to
the shape of $\E[\d^2_\st](\th)$ in the parameter space. If the shape at the minimizer is flatter,
then $C_3$ is smaller. The constant $C_2$ decreases as $\eta$ decreases, so the first term increases
as $\eta$ decreases, while the last term $O(\e^2\E[\d^2_\st]\eta^{C_3})$ does the opposite. This
suggests that one should set the learning rate $\eta$ large at first, making the exponential decay
faster. As the training progresses, $\eta$ should be reduced to make the final error smaller.





\section{Numerical examples}
\label{sec: numerics}

Code for reproducing these experiments can be found in the supplementary material. Due to space constraints, full details of the experiments can be found in Appendix \ref{experiment details}.

In each of the settings below, we test the efficacy of learning $\Q$ via SC and BFF. We test the generalized version of BFF specified by equation (\ref{eq: nbff}). The label nBFF in the plots corresponds to to using the estimate $\hF^n$ from equation (\ref{eq: nbff}); 1BFF corresponds to the standard BFF algorithm (algorithms \ref{algo: bff para} and \ref{algo: bff tab}). In each case, we use the uniform weights $\a_i = 1/n$. When applicable, we also compare to US and PD. (For the full definition of the PD algorithm, see Appendix \ref{appendix: pd}.)

\subsection{Continuous state space} \label{sec: nn bff}
We consider an MDP with continuous state space $\S = [0,2\pi).$ The transition dynamics are 
\begin{equation*}
    \dsm = \am \e + \s \zm\se,
\end{equation*}
where $\am\in\A = \{\pm1\}$ is drawn from policy $\pi$ to be defined later and $\zm \sim
N(0,1)$. We set $\e = \frac{2\pi}{32}$ and $\s = 0.2$. The reward function is
$r(\smp,\sm,\am) = \sin(\smp)+1$.

In the first two experiments, we approximate $\Q$ with a neural network with two hidden layers. Each hidden layer contains
50 neurons and cosine activations. The NN takes a state as input and outputs a vector in $\R^{|\A|}$; the $i$-th entry of the output vector corresponds to $\Q(s, a_i)$. The CartPole experiments uses a larger network with ReLU activations.

\paragraph{$Q$-evaluation}\label{sec: nn eval}
We first estimating $Q^\pi$ for the fixed policy $\pi(a|s) = 1/2+a\sin(s)/5$.
The results are plotted in Figure \ref{fig: nn eval}. BFF exhibits superior performance compared to SC and PD, with only slightly worse
performance than the (impractical) US algorithm.
\begin{figure}[h]
\begin{centering}
  \includegraphics[width=\linewidth]{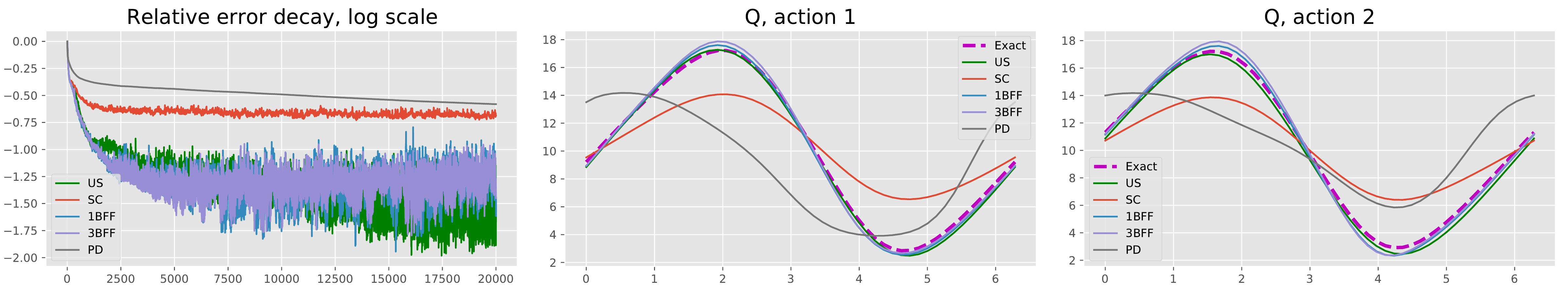}
  \caption{Results of each method for fixed-policy $Q$-evaluation. We plot the best result out of 10 runs for PD. The BFF algorithm performs better than both SC and PD. Changing the number of future steps used to compute the BFF approximation does not have a large impact on its performance in this case.}
  \label{fig: nn eval}
\end{centering}
\end{figure}

\paragraph{$Q$-control}\label{sec: nn control}
In the control case, we use a fixed behavior policy to generate the training trajectory. At each step, the behavior policy samples an action uniformly at random, i.e. $\pi(a|s) = 1/2$ for all $a\in \A$ and $s\in \S$.
The results are shown in Figure \ref{fig: nn control}. Again, BFF has comparable performance to SC and outperforms both SC and PD.

\begin{figure}[h]
\begin{center}
  \includegraphics[width=\linewidth]{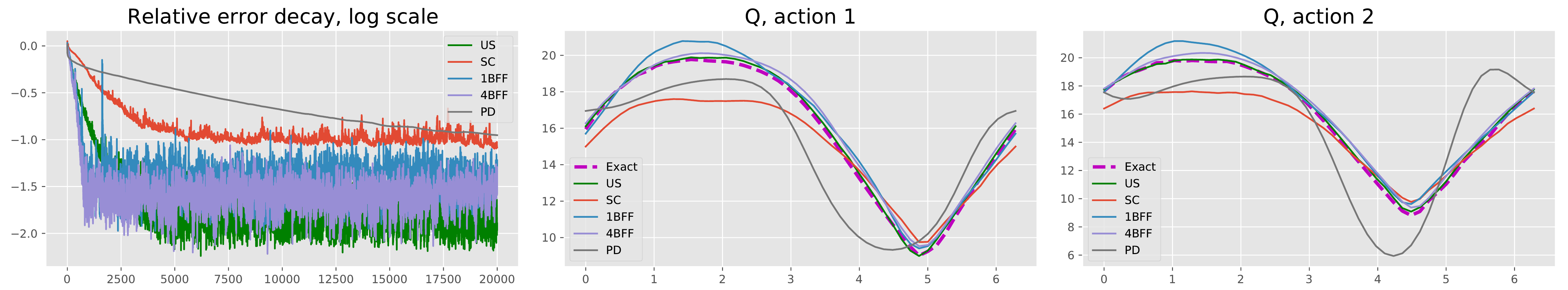}
  \caption{Results of each method for $Q$-control. We plot the best result out of 10 runs for PD. As before, the more accurate gradient estimate from BFF improves our learned approximation for $Q$. In this case, the variance reduction obtained from 4 future steps improved BFF's performance even more, giving results comparable to US.}
  \label{fig: nn control}
\end{center}
\end{figure}

\paragraph{CartPole}
We tested the BFF algorithm on the CartPole environment from OpenAI gym \cite{gym}. It is straightforward to modify BFF for use in conjunction with adaptive SGD algorithms such as Adam \cite{adam}, and we use BFF with Adam for this experiment. The results are plotted in Figure \ref{fig: cartpole}. BFF reaches the max reward (200) faster than SC and achieves it with greater
regularity throughout the training process. In contrast to both of these methods, the PD method fails to converge even after an extensive hyperparameter search.

\begin{figure}[h]
\begin{center}
  \includegraphics[width=\linewidth]{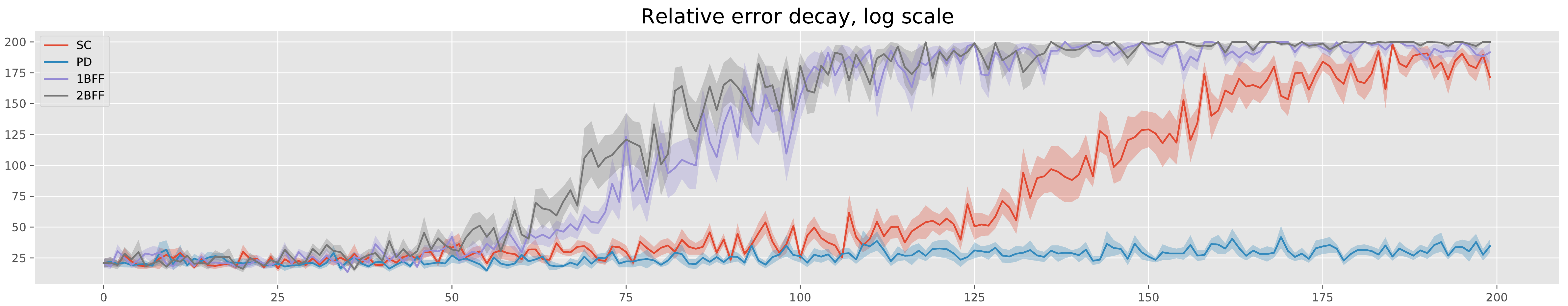}
  \caption{Reward per training episode for the CartPole experiment. BFF is the first to reach the
    maximum reward and achieves it more consistently than sample-cloning. It achieves slightly better performance using 2 future steps (2BFF in the plot). Despite an extensive
    hyperparameter search, PD was not able to learn an effective policy.}
  \label{fig: cartpole}
\end{center}
\end{figure}

\subsection{Tabular case} \label{sec: tab bff}
We next consider an MDP with a discrete state space $\S = \{\frac{2\pi k}{n}\}_{k=0}^{n-1}$ and $n =
32$. The transition dynamics are given by
\begin{equation}
    \Delta s_m = \frac{2\pi}{n} a_m \e + \sigma Z_m \sqrt{\e},
\end{equation}
where $a_m\in\A = \{\pm 1\}$ is drawn from the policy $\pi(a|s) = 1/2 + a\sin(s) / 5$ and $Z_m\sim
N(0, 1)$. We then set $s_{m+1} = \argmin_{s\in\S} |s_m + \Delta s_m - s|$. For the experiment
below, $\sigma = 1$ and $\e = 1$.
The results are plotted in Figure \ref{fig: tabular}. In this case, BFF is nearly indistinguishable
from training via US.  Due to space constraints and its similarity to the previous experiments, we
defer the case of tabular $Q$-control to the appendix.


\begin{figure}[h]
  \begin{centering}
    \includegraphics[width=\linewidth]{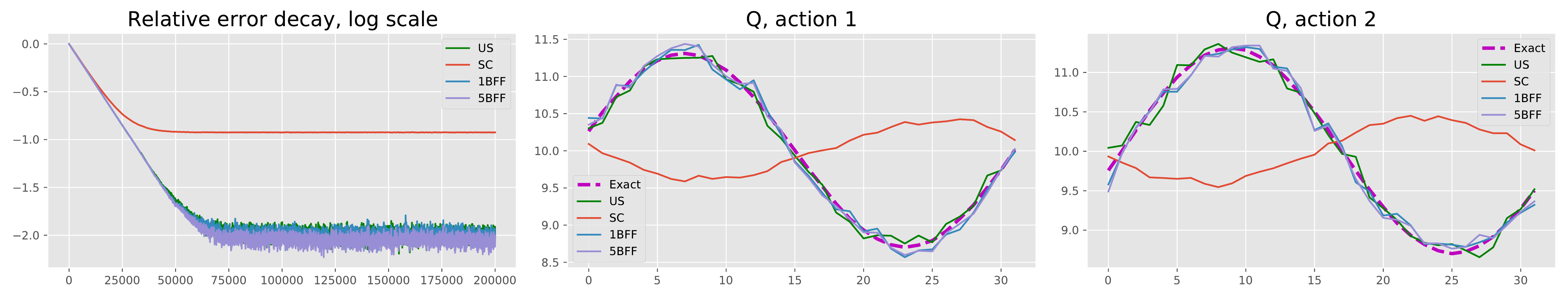}
    \caption{Results of each method for fixed-policy $Q$-evaluation in the tabular case. BFF gives a
      better estimate for the gradient than SC, leading to improved performance. BFF's performance does not change significantly with the number of future steps in this case. Note that the PD method does not apply to this case.}
    \label{fig: tabular}
  \end{centering}
\end{figure}

\section{Conclusion}
In this paper, we show that BFF has an advantage over other BRM algorithms for model-free RL
problems with continuous state spaces and smooth underlying dynamics. We also prove that the
difference between the BFF algorithm and the uncorrelated sampling algorithm first decays
exponentially and eventually stabilizes at an error of $O(\e\d_\st)$, where $\d_\st$ is the smallest
Bellman residual that US can achieve.


\section{Broader Impact}
The main societal impact of deep reinforcement learning has been its ability to automate ever more complicated tasks. Recent advances in driverless vehicles \cite{sallab2017} and automated control of robots \cite{gu2017} use deep $Q$-function approximations to learn an optimal policy. Our work on BFF contributes directly to improving the capabilities of automation.

Automation provides clear economic and utilitarian benefits. Well-designed robot or AI workers make fewer mistakes, produce greater output, and, in the long run, may cost less than their human counterparts. This leads to greater economic productivity and technological advances \cite{carlsson2012}.

Increased automation is not without its risks. As AI capabilities improve, large sections of the population may face unemployment \cite{leontief1986}. Members of underprivileged classes will likely be disproportionately affected by the decreased availability of low-skill jobs, while simultaneously having less access to the benefits automation provides. As the power of AI increases, so to does our understanding of the unintended consequences. For instance, the recent work of \cite{bissell2020} studies these effects in the case of autonomous vehicles.

BFF is a tool which can facilitate advances in science and technology, and the exacerbation of social inequality is an inherent risk of any new technology. It is via an ethical application of these new discoveries that society can realize the greatest benefit.

\newpage

\bibliographystyle{plainnat}

\appendix
\section*{Appendices}
\addcontentsline{toc}{section}{Appendices}
\renewcommand{\thesubsection}{\Alph{subsection}}

\subsection{Extension and Proof of Lemma \ref{lemma: diff uc-gd BFF short}} 
\label{appendix: lemma 1}
\setcounter{equation}{0}
\setcounter{theorem}{0}
\renewcommand\theequation{A.\arabic{equation}}
\renewcommand\thetheorem{A.\arabic{theorem}}

\begin{lemma}[Extension of Lemma \ref{lemma: diff uc-gd BFF short} for $Q$-evaluation]
\label{lemma: diff uc-gd BFF}
If \ $\dsp \sup_{s\in\S, \th\in\R^\dth}\lv\pt_s\E_a[ \nb_\th\Q(s, a;\th)|s] \rv\leq C$ and $\dsp
\sup_{s\in\S, a\in\A}\lv\E_a[\mu(s,a)|s] - \mu(s,a) \rv \leq C$ a.s., then the difference between the
gradients of the US and BFF algorithms for $Q$-evaluation is bounded by
\begin{equation*}
  \begin{aligned}
    &\lv\E[\hF] - \E[F]\rv\leq \g C^2\E\lv \d (\e + o(\e))\rv;
  \end{aligned}
\end{equation*}
In addition, if $\lv \E_a [\Q(s,a;\th)] - Q(s,a;\th)\rv,\lv \E_a [\nb_\th\Q(s,a;\th)] - \nb_\th
Q(s,a;\th)\rv,\lv\mu(s,a) - \mu(s,a')\rv,$ $\lv r(s,s,a)\rv \leq C$ for a.s. $\forall s\in \S,
a\in\A, \th\in \R^\dth$, then the difference between the variances can also be bounded by
\begin{equation*}
  \begin{aligned}
    &\lv \V[\hF] - \V[F]\rv \leq   O(\e),
  \end{aligned}
\end{equation*}
where $\V$ stands for the variance and 
\begin{equation*}
\begin{aligned}
    F &= j(\sm,\am,\smp;\th)\nb_\th j(\sm,\am,\smp';\th),\\
    \hF &= j(\sm,\am,\smp;\th)\nb_\th j(\sm,\am,\sm+\dsmp;\th),\\
    j&(\sm, \am, \smp;\th_m) = r(\smp,\sm,\am)+\g\int \Q(\smp,a;\th)\pi(a\vert\smp)da\  - \Q(\sm, \am;\th),\\
    \d&(\sm, \am;\th) = \E\l[ \l. j(\sm, \am, \smp;\th_m) \r\vert\sm, \am \r].
\end{aligned}
\end{equation*}
Note that the above form also works for the discrete action spaces. Specifically, $\pi(a|s)da=\dsp
\sum_{a_i\in\A}\pi(a_i|s)\d_{a_i}(a)da$ in the discrete action space, where $\d_{a_i}(a)$ is the
Dirac delta function.
\end{lemma}

\begin{proof}
The expectation of the US gradient is
\begin{equation}
    \label{eq: exp uncor}
    \E[F] = \E[\E[j|\sm,\am]\E[\nb_\th j' |\sm,\am ]] = \E [\d(s, a; \th) \nb_\th\d(s,a;\th)],
\end{equation}
with $j' = j(\sm, \am, \smp;\th)$. The expectation of the BFF gradient is
\begin{equation}
  \label{eq: exp BFF}
  \begin{aligned}
    &\E[\hF] = \E \l[ \E\l[j\nb_\th \hj\vert \sm ,\am \r] \r] =& \E \l[\d(s, a; \th) \E\l[\nb_\th \hj \vert \sm,\am\r] \r],
  \end{aligned}
\end{equation}
with $\hj = j(\sm, \am, \sm+\dsmp;\th)$. By subtracting the two gradients in \eqref{eq: exp uncor}
and \eqref{eq: exp BFF}, we see that the difference between the BFF and US gradients is
\begin{equation}
\label{eq: er_1}
\begin{aligned}
    &\E[\hF] - \E[F]
    =\E\l[ \d(\sm, \am)\E\l[\l. \nb_\th \hj - \nb_\th j' \r\vert\,\sm , \am \r]\r]
\end{aligned}
\end{equation}
For notational convenience, in what follows we drop the explicit dependence of $\Q$ on $\th$. All of
the gradients $\nb$ are taken with respect to $\th$. For ease of exposition, we consider a
one-dimensional state space $\S$. It is straightforward to generalize to the multi-dimensional
case. Using a Taylor expansion, we can expand $\nb\Q(\smp,a)\pi(a\vert\smp)$ around
$\nb\Q(\sm,a)\pi(a\vert\sm)$ by
\begin{equation*}
\begin{aligned}
&\nb\Q(\smp,a)\pi(a\vert\smp) \\
=& \nb\Q(\sm,a)\pi(a\vert\sm) + \pt_s(\nb\Q(\sm,a)\pi(a\vert\sm))\dsm + \frac12\pt_s^2(\nb\Q(\sm,a)\pi(a\vert\sm))\dsm^2. 
\end{aligned}
\end{equation*}
Substituting $\dsm = \mu(\sm,\am)\e + \s\zm\se$ yields \begin{equation}
\label{eq: er_2}
\begin{aligned}
    &\nb_\th j' = \g\int \nb\Q(\smp,a)\pi(a\vert\smp)da  - \nb\Q(\sm,\am)\\
    =&\underbrace{ \g\int \nb\Q(\sm,a)\pi(a\vert\sm)da  - \nb\Q(\sm,\am)}_{f_0}\\
    &+ \underbrace{\l(\g\int \pt_s(\nb\Q(\sm,a)\pi(a\vert,\sm)) da\r) \mu(\sm,\am)}_{f_1}\e \\
    &+ \underbrace{\l(\g\int \pt_s(\nb\Q(\sm,a)\pi(a\vert,\sm))da\r)\s}_{f_2}\zm\se \\
    &+\underbrace{ \l(\g\int \pt^2_s(\nb\Q(\sm,a)\pi(a\vert\sm) )da\r)\s^2}_{f_3} \zm^2\e + o(\e).
\end{aligned}
\end{equation}
Similarly, we can expand $\nb\Q(\sm+\dsmp,a)\pi(a\vert\sm+\dsmp)$ around $\nb\Q(\sm,a)\pi(a\vert\sm)$. This yields
\begin{equation*}
\begin{aligned}
&\nb\Q(\smp,a)\pi(a\vert\smp) \\
=& \nb\Q(\sm,a)\pi(a\vert\sm) + \pt_s(\nb\Q(\sm,a)\pi(a\vert\sm))\dsmp + \pt_s^2(\Q(\sm,a)\pi(a\vert\sm))\dsmp^2.
\end{aligned}
\end{equation*}
By Taylor expanding $\mu(\smp, \amp)$ around $\mu(\sm, \amp)$ and using the fact that $\dsm = O(\se)$, we see that
\begin{align*}
    \mu(\smp,\amp) &= \mu(\sm, \amp) + \pt_s\mu(\sm,\amp)\dsm + O(\dsm^2) \\
    &= \mu(\sm, \amp) + o(1).
\end{align*}
Substituting this into the expression for $\dsmp$ yields
$$\dsmp = \mu(\smp,\amp)\e + \s\zmp\se = \mu(\sm,\amp)\e + \s\zmp\se + o(\e).$$
Combining this expression for $\dsmp$ with the Taylor expansion of $\nb \Q$, we conclude that
\begin{equation}
\label{eq: er_3}
\begin{aligned}
    &\nb_\th \hj = \g\int \nb\Q(\sm+\dsmp,a)\pi(a\vert\sm+\dsmp)da  - \nb\Q(\sm,\am)\\
    = & f_0 + \underbrace{\l(\g\int \pt_s(\nb\Q(\sm,a)\pi(a\vert\sm))da\r)\mu(\sm,\amp)}_{\h{f}_1}\e + f_2\zmp\se+f_3\zmp^2\e+o(\e).
\end{aligned}
\end{equation}
It follows that 
\begin{equation*}
\begin{aligned}
    &\E\l[\l. \nb \hj - \nb j' \r\vert\,\sm , \am \r] \\
    =& \E[(\hf_1 - f_1)\e | \sm, \am] + f_2\E[\zmp - \zm|\sm, \am]\se + f_3\E[\zmp^2 - \zm^2|\sm, \am]\e + o(\e) \\
    =&\E[(\hf_1 - f_1)\e | \sm, \am] \e +  o(\e)\\
    =&\g \l(\int \pt_s(\nb\Q(\sm,a)\pi(a\vert\sm))da\r) \E[\mu(\sm,\amp) - \mu(\sm,\am)|\sm, \am] \e + o(\e) .
\end{aligned}
\end{equation*}
Recall the assumptions of the lemma:
\begin{equation*}
\begin{aligned}
     &\lv\pt_s\E_{a}[\nb_\th \Q(s,a;\th)|s] \rv \leq C, \qd \forall s\in\S, \th\in\R^\dth; \\
     &\lv\E_a[\mu(s,a)|s] - \mu(s,a) \rv \leq C,  \qd \forall s\in\S, a\in\A,.
\end{aligned}
\end{equation*} 
Using these inequalities, we find
\begin{equation*}
\begin{aligned}
    &\E\l[\l. \nb \hj - \nb j' \r\vert\,\sm , \am \r] 
    \leq \g C^2\e + o(\e).
\end{aligned}
\end{equation*}
Substituting the above inequality into \eqref{eq: er_1} finally yields 
\begin{equation*}
\begin{aligned}
    \E[\hF -  F] = \g C^2\E\l[\lv \d(\e+o(\e) )\rv\r]
\end{aligned}
\end{equation*}
which completes the proof for the first part of the lemma.

We now bound the difference of the variance. By the definition of $F, \hF$ in \eqref{def of F eval},
\eqref{def of hF eval}, we have,
\begin{equation*}
\begin{aligned}
     &\lv \V[\hF] - \V[F]\rv \\
     =& \E[j^2((\nb_\th \hj)^2 - (\nb_\th j')^2)] -\l(\E[j\nb_\th \hj]^2 - \E[j\nb_\th j']^2\r)\\
     =&\underbrace{\E\l[\E\l[j^2\vert \sm,\am\r] \E\l[(\nb_\th \hj)^2 - (\nb_\th j')^2\vert \sm,\am\r]\r]}_{I}\\ &-\underbrace{\l(\E[\E[j\vert\sm,\am]\E[\nb_\th \hj\vert\sm,\am]]^2 - \E[\E[j\vert\sm,\am]\E[\nb_\th j'\vert\sm,\am]]^2\r)}_{II}.
\end{aligned}
\end{equation*}
Using the same approximations of $\nb_\th\hj$, $\nb_\th j'$ as in \eqref{eq: er_2}, \eqref{eq: er_3} gives
\begin{equation*}
\begin{aligned}
    \nb_\th \hj - \nb_\th j' = &(f_1 - \hf_1)\e + f_2(\zmp - \zm)\se + f_3(\zmp^2-\zm^2)\e + o(\e)\\
    \nb_\th \hj + \nb_\th j' = &2f_0 + (f_1+\hf_1)\e + f_2(\zmp+\zm)\se + f_3(\zmp^2+\zm^2)\e + o(\e).
\end{aligned}
\end{equation*}
It follows that
\begin{equation*}
\begin{aligned}
    \E[(\nb_\th \hj)^2 - (\nb_\th& j')^2\vert \sm,\am] = \E[(\nb_\th \hj - \nb_\th j')(\nb_\th \hj + \nb_\th j')\vert \sm,\am]\\
    =&\E[2f_0(f_1 - \hf_1)\e + 2f_0f_2(\zmp - \zm)\se + 2f_0f_3(\zmp^2-\zm^2)\e \\
    &+f_2^2(\zmp^2-\zm^2)\e + o(\e)\vert \sm,\am]
    \\
    =&2f_0(f_1 - \hf_1)\e + o(\e),
\end{aligned}
\end{equation*}
Again, using a Taylor expansion, we can approximate $j$ by
\begin{equation}
\label{eq: er_4}
\begin{aligned}
    j =& \underbrace{r + \g\int \Q\pi da  - \Q}_{g_0}+ \underbrace{\l(\pt_s r  +\g\int \pt_s(\Q\pi) da\r)\mu }_{g_1}\e\\
    &+ \underbrace{\l(\pt_sr + \g\int \pt_s(\Q\pi) da \r)\s}_{g_2}\zm\se +\underbrace{ \l(\pt_s^2r + \g\int \pt^2_s(\Q\pi)da\r) \s^2}_{g_3}\zm^2\e  +o(\e),\\
\end{aligned}
\end{equation}
where we abbreviate $r(\sm,\sm,\am), \Q(\sm,a), \pi(a\vert\sm),$ and $\mu(\sm,\am)$ by $r, \Q, \pi,$ and $\mu$, respectively. It follows that
\begin{equation*}
\begin{aligned}
    I = \E\l[j^2\E[(\nb_\th \hj)^2 - (\nb_\th j')^2\vert \sm,\am]\r] = 2\E[g_0^2f_0(f_1 - \hf_1)]\e + o(\e).
\end{aligned}
\end{equation*}
Furthermore, we have
\begin{equation*}
\begin{aligned}
     \underbrace{\E[\nb_\th j' \vert\sm,\am]}_{\circled{1}}  =& f_0 + f_1 \e + f_3\e  + o(\e), \\
    \underbrace{\E[\nb_\th \hj \vert\sm,\am]}_{ \circled{2}} =& f_0 + \hf_1 \e + f_3\e + o(\e), \\
    \underbrace{\E[j\vert\sm,\am]}_{ \circled{3}}=& g_0 + g_1\e + g_3\e + o(\e).
\end{aligned}
\end{equation*}
Combining these expressions shows
\begin{equation*}
\begin{aligned}
    \E[\circled{2} \circled{3} ]^2 =& (\E[f_0g_0] + \E[f_0(g_1 + g_3) + g_0f_3)]\e + \E[g_0\hf_1] \e + o(\e))^2 \\
    =& \E[f_0g_0]^2 +  2\E[f_0g_0]\E[f_0(g_1 + g_3) + g_0f_3)]\e + 2\E[f_0g_0] \E[g_0\hf_1] \e +o(\e) \\
    \E[\circled{1} \circled{3} ]^2 
    =& \E[f_0g_0]^2 +  2\E[f_0g_0]\E[f_0(g_1 + g_3) + g_0f_3)]\e + 2\E[f_0g_0] \E[g_0f_1] \e +o(\e)\\
\end{aligned}
\end{equation*}
which in turn yields
\begin{equation*}
\begin{aligned}
   II =\E[\circled{2} \circled{3} ]^2 - \E[\circled{1} \circled{3} ]^2 = \E[f_0g_0] \E[g_0(\hf_1 - f_1)] \e + o(\e).
\end{aligned}
\end{equation*}
Combining $I$ and $II$, we see that 
\begin{equation*}
\begin{aligned}
     &\lv \V[\hF)] - \V[F]\rv = I - II \leq 2\text{Cov}(f_0g_0, g_0(\hf_1 - f_1)) \e + o(\e) \leq O(\e)
\end{aligned}
\end{equation*}
as long as the covariance of $f_0g_0$ and $g_0(\hf_1 - f_1)$ is bounded. But since $f_0, g_0,\hf_1 - f_1$ are all bounded by the conditions in the second part of the lemma, $\text{Cov}(f_0g_0, g_0(\hf_1 - f_1))$ must be bounded as well. This concludes the proof.
\end{proof}

\begin{lemma}[Extension of Lemma \ref{lemma: diff uc-gd BFF short} for $Q$-control]
\label{lemma: diff gd BFF control}
Let $\dsp f(s;\th) = \max_{a'\in\A} \Qst(s,a;\th)$. Suppose that $f(s; \th)$ is continuous in $s\in\S$ and that
$\pt_sf(s;\th), \pt_s^2f(s;\th)$ exist almost surely. Further assume that $\dsp
\sup_{s\in\S,\th\in\R^\dth}|\pt_s\nb_\th f(\sm;\th) |, \sup_{s\in\S,a\in\A}|\E_a[\mu(s,a) |
  s]-\mu(s,a) |\leq C$ a.s. Then the difference between the gradients in US and BFF
for $Q$-control is bounded by
\begin{equation*}
\begin{aligned}
     &\lv\E[\hF] - \E[F]\rv\leq \g C^2\E\lv \d (\se  + O(\e))\rv,
\end{aligned}
\end{equation*}
In addition, if $\dsp \lv \max_a Q(s,a;\th) - Q(s,a;\th)\rv,\lv \nb_\th\max_a\Q(s,a;\th) - \nb_\th
Q(s,a;\th)\rv,\\ \lv\mu(s,a) - \mu(s,a')\rv, \lv r(s,s,a)\rv \leq C$ almost surely over $s\in \S, a\in\A,
\th\in \R^\dth$, then
\begin{equation*}
\begin{aligned}
     &\lv \V[\hF] - \V[F]\rv \leq   O(\se).
\end{aligned}
\end{equation*}
Here $F, \hF$ are the same as in Lemma \ref{lemma: diff uc-gd BFF} with $j$ and $\delta$ replaced by
\begin{equation}
\begin{aligned}
    j&(\sm,\am,\smp;\th) = r(s_{m+1}, s_m, a_m) + \g \max_a Q^*(s_{m+1}, a; \th) - Q^*(s_m, a_m; \th);\\
    \d &= \E\l[r(s_{m+1}, s_m, a_m) + \g \max_a Q^*(s_{m+1}, a; \th) - Q^*(s_m, a_m; \th) \rvert \sm, \am\r].
\end{aligned}
\end{equation}
\end{lemma}
\begin{proof}

The difference between the two algorithms is 
\begin{equation*}
\begin{aligned}
    &\E[\hF] - \E[F]
    =\E\l[ \d(\sm, \am)\E\l[\l. \nb_\th j' - \nb_\th \hj \r\vert\,\sm , \am \r]\r]
\end{aligned}
\end{equation*}
where
\begin{equation*}
\begin{aligned}
    &\nb_\th j' = \nb_\th j(\sm, \am, \smp;\th),\qd \nb_\th \hj = \nb_\th j(\sm, \am, \sm+\dsmp;\th)
\end{aligned}
\end{equation*}
and
\begin{equation*}
\begin{aligned}
    &\nb_\th j(\sm, \am, \smp;\th) = \nb_\th \max_{a'\in\A}\Qst(\smp,a';\th) - \nb_\th \Qst(\sm,\am;\th).
\end{aligned}
\end{equation*}
Since we have assumed that $f(s;\th) = \max_{a'\in\A} \Qst(s,a;\th)$
is continuous in $s\in\S$ and that $\pt_sf(s;\th), \pt_s^2f(s;\th)$ exist almost surely, we can write $\nb_\th j$ as
\begin{equation*}
\begin{aligned}
    &\nb_\th j(\sm, \am, \smp;\th) = \nb_\th f(\smp;\th) - \nb_\th \Qst(\sm,\am;\th).
\end{aligned}
\end{equation*}
Similarly to the proof of Lemma \ref{lemma: diff uc-gd BFF}, we use a Taylor expansion:   
\begin{equation}
\label{temp2}
\begin{aligned}
    \nb_\th j' 
    =&\underbrace{ \g \nb f(\sm) - \nb\Q(\sm,\am)}_{f_0}+ \underbrace{\g \pt_s\nb f(\sm) \mu(\sm,\am)}_{f_1}\e \\
    &+ \underbrace{\g\pt_s\nb f(\sm)\s}_{f_2}\zm\se +\underbrace{ \g\pt_s^2\nb f(\sm)\s^2}_{f_3} \zm^2\e + o(\e),\\
    \nb_\th \hj
    =&f_0+ \underbrace{\g \pt_s\nb f(\sm) \mu(\sm,\amp)}_{\hf_1}\e + f_2\zmp\se +f_3 \zmp^2\e + o(\e).
\end{aligned}
\end{equation}
Using the expressions from \eqref{temp2}, we see that
\begin{equation*}
\begin{aligned}
    E[\hF] -  \E[F]=& E\l[\d \E[f_1 - \hf_1| \sm,\am] \r]\e + o(\e)\\
    =& E\l[\d\g \pt_s\nb_\th f(\sm)\l( \E[\mu(\sm,\amp)| \sm,\am] - \mu(\sm,\am)\r)\r]\e + o(\e).
\end{aligned}
\end{equation*}
Since 
\begin{equation*}
\begin{aligned}
    &\E[\mu(\sm,\amp)| \sm,\am] - \mu(\sm,\am) = \int \mu(\sm,a)\pi(a|\smp) da - \mu(\sm,\am) \\
    =& \int \mu(\sm,a)\l(\pi(a|\sm) +\pt_s\pi(a|\sm)(\mu(\sm,\am)\e+\s\zm\se)\r) da - \mu(\sm,\am) + O(\e)\\
    =&\int \mu(\sm,a)\pi(a|\sm)da - \mu(\sm,\am) + O(\se),
\end{aligned}
\end{equation*}
we have
\begin{equation*}
\begin{aligned}
    &E[\hF] - E[F] =\E\l[\d \g\l[\pt_s\nb_\th f(\sm)\l(\int \mu(\sm,a)\pi(a|\sm)da - \mu(\sm,\am)\r)\r]\r]\e + o(\e).
\end{aligned}
\end{equation*}
Since we have additionally assumed that
\begin{equation*}
\begin{aligned}
    &\sup_{s\in\S,\th\in\R^\dth}|\pt_s\nb_\th f(\sm;\th) |, \sup_{s\in\S,a\in\A}|\E_a[\mu(s,a) | s] - \mu(s,a) |\leq C,
\end{aligned}
\end{equation*}
it follows that
\begin{equation*}
\begin{aligned}
    &E[\hF] - E[F] \leq \g C^2\E[\d (\e + o(\e))]
\end{aligned}
\end{equation*}
as desired.

We next bound the difference of the variance. We have 
\begin{equation*}
\begin{aligned}
    &\V[\hF] - \V[F] = \E[j^2((\nb_\th \hj)^2 - (\nb_\th j')^2)] -\l(\E[j\nb_\th \hj]^2 - \E[j\nb_\th j']^2\r).
\end{aligned}
\end{equation*}
Substituting the Taylor expansions of $\nb_\th j', \nb_\th \hj$ from \eqref{temp2}, we obtain
\begin{equation}
\label{temp3}
\begin{aligned}
    j =& \underbrace{r(\sm,\sm,\am) + \g f(\sm)  - \Q(\sm,\am)}_{g_0}+ \underbrace{\l(\pt_s r(\sm)  +\g\pt_sf(\sm)\r)\mu }_{g_1}\e\\
    &+ \underbrace{\l(\pt_s r(\sm) + \g\pt_sf(\sm)\r)\s}_{g_2}\zm\se +\underbrace{ \l(\pt_s^2r(\sm) + \g \pt^2_sf(\sm)\r) \s^2}_{g_3}\zm^2\e  +o(\e).\\
\end{aligned}
\end{equation}
Following steps similar to the proof of Lemma \ref{lemma: diff uc-gd BFF}, we arrive at
\begin{equation*}
\begin{aligned}
    \V[\hF] - \V[F] = O(\e),
\end{aligned}
\end{equation*}
provided that $g_0, f_0, \hf_1 - f_1$ are all bounded. The boundedness of these quantities is
precisely the second set of assumptions in the lemma, so we are done.

\end{proof}

\subsection{Extension and proof of Theorem \ref{thm: diff of pdf short}}
\label{appendix: thm}
\setcounter{equation}{0}
\setcounter{theorem}{0}
\renewcommand\theequation{B.\arabic{equation}}
\renewcommand\thetheorem{B.\arabic{theorem}}

Since the continuous evolution of the parameters satisfies \eqref{eq: SDE}, the p.d.f. of the parameters in the optimization process satisfies the following two equations:
\begin{align}
     \text{Uncorrelated:}\qd &\pt_tp = \nb\cdot\l[ \E [F] p + \frac{\eta}2\nb\cdot\l(\V[F]p\r)\r];\label{eq: pdf uncor}\\
    \text{BFF:}\qd &\pt_t\hp = \nb\cdot\l[ \E [\hF] \hp + \frac{\eta}2\nb\cdot\l(\V[\hF]\hp\r)\r];\label{eq: pdf BFF}
\end{align}
Since $\E[F] = \nb_\th\E[\d^2]$ ($\d$ denotes the Bellman residual) 
and we have assumed $\V[F] \equiv \xi$, it is easy to check that the steady state of \eqref{eq: pdf uncor} is 
\begin{equation}
    \label{def pinf}
    \pinf = \frac1Ze^{- \E[\beta\d^2]}, \qd \b = \frac{2}{\eta\xi},
\end{equation}
where $Z = \int e^{-\b\E[\d^2]}d\th$ is a normalizing constant. The difference of the p.d.f. $\hd = p - \hp$ satisfies
\begin{align}
     &\pt_t\hd = \nb\cdot\l[ \E [F] \hd + \frac{\eta}2\nb\cdot\l(\V[F]\hd\r)\r] + \nb\cdot\l[ \l(\E[F] - \E[\hF]\r)\hp + \frac{\eta}2\nb\cdot\l(\l( \V[F] - \V[\hF]\r)\hp\r)\r].\label{eq: diff BFF}
\end{align}

The proof of Theorem \ref{thm: diff of pdf short} is based on Corollary \ref{coro} and Lemma \ref{lemma: d p^2}, as well as the following assumptions on $\E[\d^2]$ and $\th\in\R^\dth$. We assume that either
\begin{equation}
    \label{two condition}
    \begin{aligned}
      &1)\qd \lim_{|\th|\to\infty} \E[\d^2] \to \infty \qd\text{and}\qd \int e^{-\E[\d^2]} < \infty,\\
      &2)\qd \lim_{|\th|\to\infty} \l(\frac{|\nb\E[\d^2]|}{2} - \D \E[\d^2] \r) = +\infty,
    \end{aligned}
\end{equation}
or $\th\in\O\subset\R^\dth$ 
is in a compact set. These assumptions ensure that the probability measure $\pinf$ satisfies the Poincare inequality
\begin{equation}
  \label{poincare ineq}
  \int f^2 \pinf d\th \leq \lam(\b) \int (\nb f)^2 \pinf d\th, \qd \qd \forall \int fd\th = 0,
\end{equation}
where $\lam(\b)$ is the Poincare constant depending on $\b$. Typically $\lam(\b)$ becomes smaller
as $\b$ becomes larger.

The following two lemmas hold for any function $\d(\th)^2$ on a compact domain
$\th\in\O\subset\R^\dth$, or on an unbounded domain $\th\in\R^\dth$ if $\lim_{|\th|\to\infty} \d(\th)^2\to
+\infty$.
\begin{lemma}
\label{lemma: Gibbs}
Let $f(\th) = \E[\d^2]$ and define $f_\st = \min f(\th)$. Suppose that $f$ has only finitely many discrete minimizers, and that all of the minima are strict. Then there exists a constant $C$ (depending on the Hessian $\nb^2f$ of $f$ at each of the minimizers) such that for $\b$ large enough, 
\begin{equation*}
    \int f(\th) e^{-\b f(\th)} d\th \leq C\l(f_\st\b^{-\frac{\dth}{2}}\r)  + C\l(\b^{-\frac{\dth+2}{2}}\r),
\end{equation*}
where $\dth$ is the dimension of $\th$.
\end{lemma}

\begin{proof}
See Appendix \ref{sec: proof of Gibbs}. 
\end{proof}

\begin{corollary}
\label{coro}
Suppose that $f(\th)$ has non-strict minima, i.e. minima at which the Hessian is not strictly positive definite. Define
$\dthst$ as
\begin{equation}
    \label{def of dthst}
    \dthst = \min\{\text{number of positive eigenvalues of }\nb_\th^2 f(\th_\st): \th_\st = \argmin f(\th)\}.
\end{equation} 
Then there exists a constant $C$ (depending on the Hessian $\nb^2f$ of $f$ at each of the minimizers) such that
\begin{equation}
\label{higher order}
    \int f(\th) e^{-\b f(\th)} d\th \leq C\l(f_\st\b^{-\frac{\dthst}{2}}\r)  + C\l(\b^{-\frac{\dthst+2}{2}}\r).
\end{equation}
\end{corollary}
\begin{proof}
The proof of the corollary is similar to the proof of Lemma \ref{lemma: Gibbs}, so we omit it here.
\end{proof}

\begin{remark}\label{rmk: dthst}
  Note that the bound \eqref{higher order} depends on $\dthst$, \emph{not} the parameter dimension
  $\dth$. When the dimension of the parameter space is high (i.e. when $\dth$ is large), it is more
  likely that there are many minima which are flat in some direction (i.e. $\nb_\th^2 f(\th_\st) $
  is a positive \emph{semi}-definite matrix at these minima). In the above, $\dthst$ denotes the
  smallest number of positive eigenvalues of $\nb_\th^2 f(\th_\st)$ among all minima. As a result,
  when the dimension $\dth$ becomes larger, the upper bound $\b^{-\frac{\dthst}{2}}$ does not
  necessarily become smaller.
  
  Furthermore, if $f_\st= 0$ (i.e. there exists $\th_\st$ such that $\Q(s,a;\th_\st)$ exactly satisfies the
  Bellman equation) then
  \begin{equation*}
    \int f(\th) e^{-\b f(\th)} d\th \leq  C\l(\b^{-\frac{\dthst+2}{2}}\r).
  \end{equation*}
\end{remark}

\begin{lemma}
\label{lemma: d p^2}
The solution to \eqref{eq: pdf uncor} (i.e. the approximate p.d.f. of US)
\begin{equation*}
  \begin{aligned}
    \int \E[\d^2] \frac{(p(t,\th)-\pinf)^2}{\pinf} d\th  \leq  C_0  e^{-b(\b)t},
  \end{aligned}
\end{equation*}
where $C_0$ is a constant depending on the initial data $(p(0,\th) - \pinf)$,
$b(\b)=\frac{2\lam(\b)^2}{C+\lam(\b)}$ with Poincare constant $\lam(\b)$ and $\dsp C
=\sup_{a,s}|\nb\d(a,s) |^2$. In addition,
\begin{equation*}
  \begin{aligned}
    \int \E[\d^2] \frac{p^2(t,\th)}{\pinf} d\th  \leq  C_0  e^{-b(\b)t} + O\l(\E[\d^2_\st]\b^{-\frac{\dthst}{2}}\r),
  \end{aligned}
\end{equation*}
where $\lam(\b)$ is the Poincare constant defined in (\ref{poincare ineq}), $\E[\d^2_\st] =
\min_{\th} \E[\d^2]$, and $\dthst$ is defined in \eqref{def of dthst} for $f = \E[\d^2]$.
\end{lemma}
\begin{proof}
See Appendix \ref{sec: proof of d p^2}.
\end{proof}

We remark that all the above lemmas and theorems work for both $\O = \R^\dth$ and compact domains
$\O$. However, the following theorem holds only for compact $\O$. Therefore,
we need a reflective boundary condition for the PDEs \eqref{eq: pdf uncor}, \eqref{eq: pdf BFF}, i.e.
\begin{equation*}
    \l.\l(\E[F]p + \frac\eta2\nb\l(\V[F]p\r) \r) \cdot \vec{n} \r\vert_{\pt \O} = 0,
\end{equation*}
and similar boundary conditions for $\hp$. It is not clear whether the compactness assumption can be
removed; we leave it for future study. In practice, the BFF algorithm still works for
unconstrained domains.

\begin{theorem}
\label{thm: diff of pdf}
The difference $\hd$ between the p.d.f.s of the US and BFF algorithms is bounded by
\begin{equation}
\label{eq: diff of pdf}
\begin{aligned}
  \ll \hd(t) \rl_\st \leq& \ll p(0) - \pinf \rl_\st e^{-\frac{\lam(\b)}4t} +
  O(\e)e^{-\frac{b(\b)}2t} + O\l(\e\sqrt{\E[\d_\st^2]}\b^{-\frac{\dthst}{4}}\r)
  \sqrt{1-e^{-\frac{\lam(\b)}2t}},
\end{aligned}
\end{equation}
where $\lam(\b)$ is the Poincare constant defined in (\ref{poincare ineq}), $b(\b)$ is the same
constant as in Lemma \ref{lemma: d p^2}, $\E[\d^2_\st] = \min_{\th} \E[\d^2]$, and $\dthst$ is
defined in \eqref{def of dthst} for $f = \E[\d^2]$.
\end{theorem}

\begin{proof}
  Observe that
  \begin{equation*}
    \begin{aligned}
    \E[\hF]\hp - \E[F]p &= \E[F]\hd + (\E[\hF] -\E[ F])\hd + (\E[\hF]-\E[F])p \\
    &= \E[F]\hd + E[\d O(\e)]\hd + E[\d O(\e)]p.
    \end{aligned}
\end{equation*}
Similarly, we have
\begin{equation*}
    \begin{aligned}
    \V[\hF]\hp - \V[F]p =  \V[F]\hd + O(\e)\hd + O(\e)p,
    \end{aligned}
\end{equation*}
Multiplying \eqref{eq: diff BFF} by $\frac{\hd}{\pinf}$, then integrating with respect to $\th$, we
have
\begin{equation*}
    \begin{aligned}
    \frac12\pt_t\ll \hd \rl^2_\st =& \underbrace{-\int \l[ \pinf \nb\l(\frac{\hd}{\pinf}\r) \r] \cdot \nb\l(\frac{\hd}{\pinf}\r)d\th }_{I}
    \underbrace{-\int \l( \E[\d O(\e)]\hd  + O(\eta\e)\lv \nb\hd\rv \r) \cdot \nb\l(\frac{\hd}{\pinf}\r) d\th}_{II}\\
    &\underbrace{ -\int \l(\E[\d O(\e)]p + O(\eta\e)\lv\nb p\rv\r) \cdot\nb\l(\frac{\hd}{\pinf}\r) d\th}_{III}.
    \end{aligned}
\end{equation*}
We proceed by bounding the terms $I-III$ separately. First, note that
\begin{equation*}
  \begin{aligned}
    I = -\int \l[\nb\l(\frac{\hd}{\pinf}\r) \r]^2\pinf d\th  \leq  -\frac12\int \l[\nb\l(\frac{\hd}{\pinf}\r) \r]^2\pinf d\th -\frac{\lam}2 \ll \hd \rl^2_\st,
  \end{aligned}
\end{equation*}
where we have used the Poincare inequality \eqref{poincare ineq}. For the second term, we have 
\begin{equation*}
  \begin{aligned}
    II \leq&  O(\e) \int \lv\hd\rv \lv\nb\l(\frac{\hd}{\pinf}\r) \rv d\th + O(\e\eta)\int \lv\nb\l(\frac{\hd}{\pinf}\pinf\r)\rv \lv\nb\l(\frac{\hd}{\pinf}\r)\rv d\th \qd \text{(boundedness of}\E\d)\\
    \leq&  O(\e)\ll \hd \rl^2_\st+\frac18\int \l[\nb\l(\frac{\hd}{\pinf}\r) \r]^2\pinf d\th + O(\e\eta)\int \lv\nb\l(\frac{\hd}{\pinf}\r)\rv^2\pinf d\th \\
    &+ O(\e\eta)\int \lv \b\E[\d\nb\d]\hd \rv \lv\nb\l(\frac{\hd}{\pinf}\r)\rv d\th \hspace{1.85in} \text{(Cauchy-Schwartz Inequality)}\\
    \leq &O(\e)\ll \hd \rl^2_\st+\frac14\int \l[\nb\l(\frac{\hd}{\pinf}\r) \r]^2\pinf d\th + O(\e\eta)\int \lv\nb\l(\frac{\hd}{\pinf}\r)\rv^2\pinf d\th+ O(\e^2\eta^2\b^2)\ll \hd \rl^2_\st.
  \end{aligned}
\end{equation*}
Since $\eta^2\b^2 = O(1)$, $O(\e^2\eta^2\b^2) = O(\e^2)$. This yields 
\begin{equation*}
    \begin{aligned}
    II \leq & O(\e)\ll \hd \rl^2_\st+\l(\frac14 + O(\e\eta)\r)\int \l[\nb\l(\frac{\hd}{\pinf}\r) \r]^2\pinf d\th.
    \end{aligned}
\end{equation*}
For the last term, by the Cauchy Schwartz inequality, 
\begin{equation*}
    \begin{aligned}
    III \leq&  O(\e^2) \int \E[\d^2] \frac{p^2}{\pinf} d\th +  \frac{1}{16}\int \lv\nb\l(\frac{\hd}{\pinf}\r) \rv^2\pinf d\th + O(\e\eta)\int \lv\nb\l(\frac{p}{\pinf}\pinf\r)\rv \lv\nb\l(\frac{\hd}{\pinf}\r)\rv d\th \qd \\
    \leq&  O(\e^2) \int \E[\d^2] \frac{p^2}{\pinf} d\th  + O(\e^2\eta^2)\int \lv\nb\l(\frac{p}{\pinf}\r)\rv^2\pinf d\th +\frac2{16}\int \l[\nb\l(\frac{\hd}{\pinf}\r) \r]^2\pinf d\th\\
    &+O(\e^2\eta^2\b^2)\int \E[\d^2|\nb\d|^2]\frac{p^2}{\pinf} d\th + \frac{1}{16}\int \lv\nb\l(\frac{\hd}{\pinf}\r) \rv^2\pinf d\th\\
    \leq &  O(\e^2) \int \E[\d^2] \frac{p^2}{\pinf} d\th  + O(\e^2\eta^2)\int \lv\nb\l(\frac{p}{\pinf}\r)\rv^2\pinf d\th +\frac3{16}\int \l[\nb\l(\frac{\hd}{\pinf}\r) \r]^2\pinf d\th.
    \end{aligned}
\end{equation*}
Combining the above three terms, we have
\begin{equation*}
    \begin{aligned}
    \frac12\pt_t \ll \hd \rl^2_\st \leq& \l(\frac1{16}- O(\e\eta)\r) \int \l[\nb\l(\frac{\hd}{\pinf}\r) \r]^2\pinf d\th - \l(\frac\lam2 - O(\e)\r)\ll \hd \rl^2_\st \\
    &+ O(\e^2)\int \E[\d^2] \frac{p^2}{\pinf} d\th  + O(\e^2\eta^2)\int \lv\nb\l(\frac{p}{\pinf}\r)\rv^2\pinf d\th .
    \end{aligned}
\end{equation*}
As long as $\e, \eta$ are small enough, and using the fact that $\nb\l(\frac{\pinf}{\pinf}\r) = 0$, we have 
\begin{equation}
\label{eq: proof hd 1}
    \begin{aligned}
    \frac12\pt_t \ll \hd \rl^2_\st \leq&  - \frac\lam4\ll \hd \rl^2_\st + O(\e^2)\int \E[\d^2] \frac{p^2}{\pinf} d\th  + O(\e^2\eta^2)\int \lv\nb\l(\frac{p - \pinf}{\pinf}\r)\rv^2\pinf d\th . 
    \end{aligned}
\end{equation}
Setting $d = p - \pinf$, it is easy to see that $d$ also satisfies \eqref{eq: pdf
  uncor}. Multiplying \eqref{eq: pdf uncor} by $\frac{p}{\pinf}$ and integrating with respect to
$\th$, we have
\begin{equation}
\label{eq: proof d 1}
    \begin{aligned}
    \frac12\pt_t \ll d \rl^2_\st =&  - \int \lv\nb\l(\frac{d}{\pinf}\r)\rv^2\pinf d\th \leq -\frac12
    \int \lv\nb\l(\frac{d}{\pinf}\r)\rv^2\pinf d\th - \frac\lam2\ll d \rl^2_\st . 
    \end{aligned}
\end{equation}
Adding equations \eqref{eq: proof hd 1} and \eqref{eq: proof d 1} gives
\begin{equation*}
    \begin{aligned}
  \frac12\pt_t \l(\ll \hd \rl^2_\st + \ll d \rl^2_\st \r)  \leq&  - \frac\lam4\l(\ll \hd \rl^2_\st + \ll d \rl^2_\st \r)  - \frac\lam4 \ll d\rl^2_\st  + O(\e^2)\int \E[\d^2] \frac{p^2}{\pinf} d\th \\
    \pt_t \l[e^{\frac\lam2t}\l(\ll \hd \rl^2_\st + \ll d \rl^2_\st \r)\r] 
    \leq& e^{\frac\lam2t}O(\e^2)\int \E[\d^2] \frac{p^2}{\pinf} d\th.  \\
    \end{aligned}
\end{equation*}
Applying Lemma \ref{lemma: d p^2}, we have
\begin{equation*}
    \begin{aligned}
    \pt_t \l[e^{\frac\lam2t}\l(\ll \hd \rl^2_\st + \ll d \rl^2_\st \r)\r] 
    \leq& O(\e^2)e^{\frac\lam2t} e^{-bt} + O\l(\e^2\E[\d^2_\st]\b^{-\frac{\dthst}{2}}\r)e^{\frac\lam2t} .     \end{aligned}
\end{equation*}
We then integrate the above inequality on both sides to obtain
\begin{equation*}
    \begin{aligned}
    &\l(\ll \hd(t) \rl^2_\st + \ll d(t) \rl^2_\st \r)  \\
    \leq& e^{-\frac\lam2t}\l(\ll \hd(0) \rl^2_\st + \ll d(0) \rl^2_\st \r) + O(\e^2)(e^{-bt} + e^{-\frac\lam2t}) + O\l(\e^2\E[\d^2_\st]\b^{-\frac{\dthst}{2}}\r) (1-e^{-\frac\lam2t}).
    \end{aligned}
\end{equation*}
Since $\hd(0) = 0$, the above inequality is equivalent to
\begin{equation*}
    \begin{aligned}
    \ll \hd(t) \rl^2_\st  \leq& e^{-\frac\lam2t}\ll d(0) \rl^2_\st + O(\e^2)e^{-bt}  + O\l(\e^2\E[\d^2_\st]\b^{-\frac{\dthst}{2}}\r) (1-e^{-\frac\lam2t})
    \end{aligned}
\end{equation*}
as desired.
\end{proof}

\subsubsection{Proof of Lemma \ref{lemma: Gibbs}}
\label{sec: proof of Gibbs}
\begin{proof}
For the unbounded domain, since $\lim_{|\th|\to\infty} f(\th) = +\infty$ and $\lim_{f\to+\infty}f
e^{-\b f} = 0$, there always exists a compact domain $\O = \{|\th| \leq M\}$ such
that $$\int_{\R^\dth\backslash\O} f(\th) e^{-\b f(\th) }d\th \leq
O\l(f(\th_\st)\b^{-\frac{\dth}{2}}\r) + O\l(\b^{-\frac{\dth+2}{2}}\r).$$

We can divide $\O$ into $\{\O_i\}_{i=1}^k$ such there is only one minimizer $\th_\st$ in each
$\O_i$, or else $f(\O_i) \equiv 0$. For this latter case, it is trivial to see that $\int_{\O_i}
f(\th) e^{-\b f(\th)} = 0$.

For the former case, notice that the integral can be separated into two parts, 
\begin{equation*}
\begin{aligned}
    &\int_{\O_1} f(\th) e^{-\b f(\th)} d\th  = \int_{|\th - \th_\st|\leq \v}f(\th) e^{-\b f(\th)} d\th  + \int_{\O_1\backslash|\th - \th_\st|\leq \v}f(\th) e^{-\b f(\th)} d\th. \\
\end{aligned}
\end{equation*}
For any $\v>0$, we can choose $\b$ large enough that the second integral will be smaller than
$O(\b^{-\frac{\dth+2}{2}})$. Since $\th_\st$ is a minimizer, $\nb f(\th_\st) = 0$ and we have
\begin{equation}
\label{proof1_0}
\begin{aligned}
    &\int_{\O_1} f(\th) e^{-\b f(\th)} d\th\\
    =& \int_{|\th - \th_\st|\leq \v}  \l(f(\th_\st)+ (\th-\th_\st)^\top\nb^2f(\th_\st)(\th-\th_\st) + O(|\th - \th_\st|^3)  \r)\\
    &\exp\l( -\b f(\th_\st) -\b (\th-\th_\st)^\top \nb^2f(\th_\st)(\th-\th_\st) -\b O(|\th - \th_\st|^3) \r)d\th + O(\b^{-\frac{\dth+1}{2}})\\
    =&f(\th_\st)\exp(-\b f(\th_\st)) \int_{|\th - \th_\st|\leq \v} \exp\l(-\b (\th-\th_\st)^\top \nb^2f(\th_\st)(\th-\th_\st) \r)d\th \\
    &+\exp(-\b f(\th_\st)) \int_{|\th - \th_\st|\leq \v} (\th-\th_\st)^\top\nb^2f(\th_\st)(\th-\th_\st)\exp\l(-\b (\th-\th_\st)^\top \nb^2f(\th_\st)(\th-\th_\st) \r)d\th\\
    &+ (\text{higher order terms in }\b).
\end{aligned}
\end{equation}
We will prove later the higher order terms are all smaller than $O(\b^{-\frac{\dth+1}{2}})$. Without
loss of generality, we assume $\nb^2f(\th_\st)$ is a diagonal matrix. (If it is not, we can simply
perform a change of basis.) Since $\th_\st$ is a local minimum, $\pt_{\th_i}^2f(\th_\st)>0$. Then
after making the change of variables $\tth = \th - \th_\st$, we have
\begin{equation}
\label{proof1_1}
\begin{aligned}
    &\int_{\O_1} f(\th) e^{-\b f(\th)} d\th \\
    =& f(\th_\st)\exp(-\b f(\th_\st))\int_{\O_1}  \prod_i\exp\l(-\b\pt^2_{\th_i}f(\th_\st) \tth_i^2\r) d\tth_1\cdots d\tth_\dth\\
    &+ \exp(-\b f(\th_\st))\int_{\O_1} \l(\sum_i\pt^2_{\th_i}f(\th_\st) \tth_i^2\r) \prod_i\exp\l(-\b\pt^2_{\th_i}f(\th_\st) \tth_i^2\r) d\tth_1\cdots d\tth_\dth + O(\b^{-\frac{\dth+1}{2}}).
\end{aligned}
\end{equation}
Since 
\begin{equation*}
\begin{aligned}
    &\int_\R \exp\l(-\b\pt^2_{\th_i}f(\th_\st) \tth_i^2\r) d\th_i = \sqrt{2\pi}(\b\pt^2_{\th_i}f(\th_\st))^{-1/2},\\
    &\int_\R \tth_i^2\exp\l(-\b\pt^2_{\th_i}f(\th_\st) \tth_i^2\r) d\th_i = \sqrt{2\pi}(\b\pt^2_{\th_i}f(\th_\st))^{-3/2},
\end{aligned}
\end{equation*}
we have,
\begin{equation*}
\begin{aligned}
    \int_{\R^\dth}  \prod_i\exp\l(-\b\pt^2_{\th_i}f(\th_\st) \tth_i^2\r) d\tth_1\cdots d\tth_\dth &= (2\pi)^{\dth/2}\b^{-\dth/2}\prod_i \l(\pt_{\th_i}^2f(\th_\st)\r)^{-1/2} \\
    &= O\l(\b^{-\frac{\dth}{2}}\r);\\
    \int_{\R^\dth} \th_i^2\prod_i\exp\l(-\b\pt^2_{\th_i}f(\th_\st) \tth_i^2\r) d\tth_1\cdots d\tth_\dth &= (2\pi)^{\dth/2}\b^{-\dth/2-1}(\pt^2_{\th_i}f(\th_\st))^{-1}\prod_j(\pt^2_{\th_j}f(\th_\st))^{-1/2} \\
    &=O\l(\b^{-\frac{\dth+2}{2}}\r).
\end{aligned}
\end{equation*}
Plugging the above estimate back to \eqref{proof1_1} and recalling that $\th_\st$ is the only
minimizer in $\O_1$, we have
\begin{equation}
\label{proof1_2}
\begin{aligned}
    &\int_{\O_1} f(\th) e^{-\b f(\th)} d\th \leq O\l(f(\th_\st)\b^{-\frac{\dth}{2}}\r)  + O\l(\b^{-\frac{\dth+2}{2}}\r). 
\end{aligned}
\end{equation}
Now we will estimate the higher order terms in \eqref{proof1_0},
\begin{equation*}
\begin{aligned}
    &(\text{higher order terms in }\b)\\ 
    =&f(\th_\st)\exp(-\b f(\th_\st)) \int_{|\tth|\leq \v} \underbrace{\exp\l(-\b \tth^\top \nb^2f(\th_\st)\tth \r)}_{\leq 1}\underbrace{\l(e^{-\beta O(|\tth|^3)} - 1\r)}_{\leq e^{\b\v^3} - 1}d\th \\
    &+\exp(-\b f(\th_\st)) \int_{|\tth|\leq \v} \underbrace{\tth^\top\nb^2f(\th_\st)\tth\exp\l(-\b \tth^\top \nb^2f(\th_\st)\tth\r)}_{\leq \v^2\max_i\{\pt_{\th_i}^2f(\th_\st)\}}\l(e^{-\beta O(|\tth|^3)} - 1\r)d\th\\
    &+ \exp(-\b f(\th_\st)) \int_{|\tth|\leq \v} \underbrace{O(|\tth|^3)}_{\leq \v^3}\underbrace{\exp\l(-\b \tth^\top \nb^2f(\th_\st)\tth-\beta O(|\tth|^3 )\r)}_{\leq 1} d\th\\
    \leq &O\l(\text{vol}\l(\{|\th|\leq \v\}\r) \l((e^{\b\v^3 - 1}) + \v^3\r)\r). 
\end{aligned}
\end{equation*}
From the above estimates, we see that as long as $\v$ is small enough, the higher order terms are
smaller than $O\l(\b^{-\frac{\dth_\st+2}2}\r)$. Since we assumed that the number of discrete
minimizers is finite, this completes the proof.

\end{proof}

\subsubsection{Proof of Lemma \ref{lemma: d p^2}}
\label{sec: proof of d p^2}
\begin{proof}
Setting $d = p - \pinf$, it is easy to see that $d$ also satisfies \eqref{eq: pdf
  uncor}. Multiplying \eqref{eq: pdf uncor} by $\E[\d^2]\frac{d}{\pinf}$ and then integrating with
respect to $\th$, we have
\begin{equation*}
    \begin{aligned}
    \frac12\pt_t\int \E[\d^2] \frac{d^2}{\pinf} d\th  \leq&  -\int \pinf\nb\l(\frac{d}{\pinf}\r) \nb\l( \E[\d^2] \frac{d}{\pinf}\r)\\
    =&  -\int \int \pinf\l[\nb\l(\frac{d}{\pinf}\r) \nb\l( \d^2 \frac{d}{\pinf}\r)\r] d\th d\mu(s,a)\\
    =& -\int \int \pinf\l[\l(\nb\l(\frac{\d d}{\pinf}\r)\r)^2 - \l(\frac{d}{\pinf}\r)^2(\nb\d)^2\r] d\th d\mu(s,a)\\
    \leq&-\int \int \pinf\l(\nb\l(\frac{\d d}{\pinf}\r)\r)^2d\th d\mu(s,a) +C\int \int\l(\frac{d}{\pinf}\r)^2\pinf d\th d\mu(s,a)\\
    \leq & -\lam \int \int \frac{(\d d)^2}{\pinf} d\th d\mu(s,a) + C \int \frac{d^2}{\pinf }d\th d\mu(s,a)\\
    \leq & -\lam\int \E[\d^2] \frac{d^2}{\pinf} d\th + C\ll d \rl^2_\st.
    \end{aligned}
\end{equation*}
Using the fact that $\frac12\pt_t\ll d \rl^2_\st \leq -\lam \ll d\rl^2_\st$, we have
\begin{equation*}
    \begin{aligned}
    \frac12\pt_t\l[\int \E[\d^2] \frac{d^2}{\pinf} d\th + \l(\frac{C}{\lam} + 1\r)\ll d \rl^2_\st\r]  
    \leq & -\lam\l(\int \E[\d^2] \frac{d^2}{\pinf} d\th  + \ll d\rl^2_\st\r)\\
    \leq &-\frac{\lam^2}{C+\lam} \l[\int \E[\d^2] \frac{d^2}{\pinf} d\th + \l(\frac{C}{\lam} + 1\r)\ll d \rl^2_\st\r].
    \end{aligned}
\end{equation*}
By Grownwall's inequality, 
\begin{equation*}
    \begin{aligned}
    \l[\int \E[\d^2] \frac{d(t)^2}{\pinf} d\th + \l(\frac{C}{\lam} + 1\r)\ll d(t) \rl^2_\st\r]  
    \leq &e^{-\frac{2\lam^2}{C+\lam}t} \l[\int \E[\d^2] \frac{d(0)^2}{\pinf}d\th + \l(\frac{C}{\lam} + 1\r)\ll d(0) \rl^2_\st\r] ,\\
    \int \E[\d^2] \frac{d(t)^2}{\pinf} d\th  
    \leq & e^{-\frac{2\lam^2}{C+\lam}t} \l[\int \E[\d^2] \frac{d(0)^2}{\pinf}d\th + \l(\frac{C}{\lam} + 1\r)\ll d(0) \rl^2_\st\r].
    \end{aligned}
\end{equation*}
which completes the first part of the proof. 

While the second part of the Lemma is obtained by inserting $p^2 = (d+\pinf)^2 \leq 2d^2 +
2(\pinf)^2$ into the following equation,
\begin{equation*}
    \begin{aligned}
    \int \E[\d^2] \frac{p(t)^2}{\pinf} d\th \leq 2\int \E[\d^2] \frac{d(t)^2}{\pinf} d\th  + 2\int \E[\d^2] \pinf d\th 
    = & C_0e^{-\frac{2\lam^2}{C+\lam}t}  + O(\b^{-\frac{\dth+2}{2}}),
    \end{aligned}
\end{equation*}
where Lemma \ref{lemma: Gibbs} is applied to the last equality. 

\end{proof}

\subsection{Difference between SC and US}
\label{sec: ds ub}
\setcounter{equation}{0}
\setcounter{theorem}{0}
\renewcommand\theequation{C.\arabic{equation}}
\renewcommand\thetheorem{C.\arabic{theorem}}

The SC parameter update is given by
\begin{equation*}
  \th_{m+1} = \th_m - \eta \tF, \qd \tF = j(\sm, \am, \smp;\th_m) \nb_\th j(\sm, \am, \smp;
  \th_m).
\end{equation*}
The definition of $j$ depends on whether we are doing $Q$-evaluation or $Q$-control:
\begin{equation*}
\begin{aligned}
  \text{$Q$-evaluation:}\qd j(\sm, \am, \smp;\th_m) = &r(\smp,\sm,\am)+\g\int \Q(\smp,a;\th)\pi(a\vert\smp)da\ \\
  &- \Q(\sm, \am;\th);\\
  \text{$Q$-control:}\qd j(\sm, \am, \smp;\th_m) =& r(\smp,\sm,\am)+\g\max_a \Q(\smp,a;\th)  - \Q(\sm, \am;\th).
\end{aligned}
\end{equation*}
The expectation of the SC gradient $\tF$ at each step is
\begin{equation}
    \label{eq: exp ds}
\begin{aligned}
    \E[\tF] = \E &\l[\E\l[ j \nb_\th j\vert \sm = s,\am = a\r] \r], 
\end{aligned}
\end{equation}
which is the gradient of the following loss function
\begin{equation}
    \label{eq: loss ds}
    \tJ(\th) = \frac12\E \l[\E\l[ \l. j^2 \r\vert \sm = s,\am = a\r] \r].
\end{equation}
Note that this is not the same as the desired objective function $J(\th) = \frac12\E[(\E[j|\sm,\am])^2]$.

\subsubsection{Difference at each step} 
In Lemmas \ref{lemma: diff uc-gd sc-gd} and \ref{lemma: control diff uc-gd sc-gd}, we prove that the
difference between the gradients used in US and SC is $O(\e)$. The constants hidden by the big-$O$
depend on the square of the diffusion $\sigma^2$. In practice, this means that SC will not converge
to a good approximation for $\Q$.

\begin{lemma}
  \label{lemma: diff uc-gd sc-gd}
  Suppose that $$\dsp \sup_{s\in\S,\th\in\R^\dth}\lv\pt_s\E_{a}[ \Q(s,a;\th)|s] \rv,
  \sup_{s\in\S,\th\in\R^\dth}\lv\pt_s\E_{a}[\nb_\th \Q(s,a;\th)|s] \rv, \dsp\sup_{s\in\S,a\in\A}
  \lv\pt_{s'}r(s,s,a) \rv \leq C $$ almost surely. Then the difference between the gradients in the
  US and SC algorithms is bounded by
  \begin{equation*}
    \begin{aligned}
      &\lv\E[\tF] - \E[F]\rv \leq 2\s^2C^2\e + o(\e);
    \end{aligned}
  \end{equation*}
  In addition, if $\lv \E_a [\Q(s,a;\th)] - Q(s,a;\th)\rv,\lv \E_a [\nb_\th\Q(s,a;\th)] - \nb_\th
  Q(s,a;\th)\rv, \lv r(s,s,a)\rv \leq C$ almost surely in $\forall s\in \S, a\in\A, \th\in \R^\dth$, then
  the difference between the variances is bounded by
  \begin{equation*}
    \begin{aligned}
      &\lv \V[\tF] - \V[F]\rv \leq   O(\e),
    \end{aligned}
  \end{equation*}
  where $\tF, F, \s$ are defined in \eqref{eq: exp ds} \eqref{eq: exp uncor} and \eqref{def of transition} respectively.
\end{lemma}

\begin{proof}
The proof is similar to the proof of Lemma \ref{lemma: diff uc-gd BFF}. Subtracting \eqref{eq: exp
  uncor} from \eqref{eq: exp ds} yields
\begin{equation}
\label{eq: error for sample-cloning}
\begin{aligned}
  &\E\l[ \tF - F \r]  = \E\l[\E[j \l( \nb j - \E[\nb j\vert\sm,\am] \r) \vert \sm,\am]\r].
\end{aligned}
\end{equation}
By the approximation of $\nb j$ in \eqref{eq: er_2}, we have
\begin{equation*}
  \begin{aligned}
    \nb j  - \E[\nb j |\sm,\am]= &f_2\zm\se + f_3(\zm^2 - 1)\e + o(\e).
  \end{aligned}
\end{equation*}
Combining this with the approximation of $j$ in \eqref{eq: er_4}  gives
\begin{equation*}
\begin{aligned}
    &\E\l[ \tF - F \r]= \E[g_2f_2\e] + o(\e) \\
    =& \g\E\l[\pt_s r \int \pt_s(\nb\Q\pi)da\r]\s^2\e + \g^2\E\l[\int \pt_s(\Q\pi)da\int \pt_s(\nb\Q\pi)da\r]\s^2\e.
\end{aligned}
\end{equation*}
Therefore, as long as, 
\begin{equation}
\label{cond_11}
\begin{aligned}
     &\pt_s\l(\E_{a}\nb \Q(s,a;\th)\r) \leq C, \qd \forall s\in\S, \th,\\
     &\pt_s\l(\E_{a} \Q(s,a;\th)\r) \leq C, \qd \forall s\in\S, \th,\\
     &\pt_{s'}r(s',s,a) \leq C, \qd \forall s',s\in\S, a\in\A,
\end{aligned}
\end{equation}
then the difference between the gradients is bounded by
\begin{equation*}
  \begin{aligned}
    \lv \E[\tF] - \E[F]\rv \leq 2\g\s^2C^2\e + o(\e)
  \end{aligned}
\end{equation*}
as desired.

Next, we bound the difference of the variance. We have
\begin{equation}
\label{eq: var2_0}
\begin{aligned}
    &\lv \V[\tF] - \V[F]\rv = \E[j^2((\nb j)^2 - (\nb j') ^2)] - \l(\E[j\nb j]^2 - \E[j\nb j' ]^2\r)\\
     = &\underbrace{ \E[\E[j^2\l((\nb j)^2 - \E[(\nb j)^2|\sm,\am]]  \r)| \sm,\am]}_{I} \\
     &- \underbrace{\l(\E[\E[j\nb j|\sm,\am] ]^2 - \E[\E[j|\sm,\am]\E[\nb j|\sm,\am]]^2\r)}_{II}.
\end{aligned}
\end{equation}
Using the approximations for $\nb j$, $j$ in \eqref{eq: er_2}, \eqref{eq: er_4}, we have 
\begin{equation*}
\begin{aligned}
    & \underbrace{\E[(\nb j)^2|\sm,\am] - (\nb j )^2 }_{\circled{1}}\\
    = &\E[f_0^2 + 2f_0f_2\zm\se + 2f_0f_1\e + (2f_0f_3+f_2^2)\zm^2\e  + o(\e)|\sm,\am]\\
    &- \l(f_0^2 + 2f_0f_2\zm\se + 2f_0f_1\e + (2f_0f_3+f_2^2)\zm^2\e  + o(\e)\r)\\
    =& -2f_0f_2\zm\se + (2f_0f_3+f_2^2)(1-\zm^2)\e  + o(\e);\\
    &\E[j^2 \circled{1} | \sm,\am]\\
    =& \E\l[ \l(g_0^2 + 2g_0g_2\zm\se + 2g_0g_1\e + (2g_0g_3+g_2^2)\zm^2\e  + o(\e) \r)\circled{1}| \sm,\am\r]\\
    =& -4g_0g_2f_0f_2\e + o(\e).
\end{aligned}
\end{equation*}
It follows that
\begin{equation}
\label{eq: var2_1}
\begin{aligned}
    &I =- \E[\E[j^2 \circled{1} | \sm,\am]]= 4\e\E[g_0g_2f_0f_2] + o(\e).
\end{aligned}
\end{equation}
Furthermore, we have
\begin{equation*}
\begin{aligned}
    &\underbrace{\E[j|\sm,\am]}_{\circled{2}} = g_0 +(g_1 + g_3)\e + o(\e),\\
    &\underbrace{\E[\nb j | \sm, \am]}_{\circled{3}}= f_0 + (f_1 + f_3)\e + o(\e),\\
    &\E[\circled{2}\circled{3}]^2 = (\E[g_0f_0] + \E[(g_0(f_1 +f_3)+f_0(g_1+g_3))\e] + o(\e))^2 \\
    =& \E[g_0f_0]^2 +2 \E[g_0f_0]\E[(g_0(f_1 +f_3)+f_0(g_1+g_3))]\e  + o(\e),
\end{aligned}
\end{equation*}
and
\begin{equation*}
\begin{aligned}
    &\underbrace{\E[j\nb j | \sm, \am] }_{\circled{4}}\\
    =& \E[f_0g_0 + f_0g_1\e + f_0g_2\zm\se + f_0g_3\zm^2\e + f_1g_0\e + f_2g_0\zm\se + f_2g_2\zm^2\e \\
    &+ g_0f_3\zm^2\e| \sm, \am] + o(\e)\\
    =&f_0g_0 + (f_0(g_1+g_3) + g_0(f_1 + f_3))\e + f_2g_2\e + o(\e)\\
    &\E[\circled{4}]^2 = \l(\E[f_0g_0] + \E[(f_0(g_1+g_3) + g_0(f_1 + f_3))]\e + \E[f_2g_2]\e + o(\e)\r)^2 \\
    =& \E[f_0g_0]^2 + 2\E[f_0g_0]\E[(f_0(g_1+g_3) + g_0(f_1 + f_3))]\e + 2\E[f_0g_0]\E[f_2g_2]\e + o(\e).
\end{aligned}
\end{equation*}
Combining these shows that
\begin{equation}
\label{eq: var2_2}
\begin{aligned}
    &II = \E[\circled{4}]^2  - \E[\circled{2}\circled{3}]^2 = 2\E[f_0g_0]\E[f_2g_2]\e + o(\e).
\end{aligned}
\end{equation}
Finally, substituting \eqref{eq: var2_1} and \eqref{eq: var2_2} into \eqref{eq: var2_0} gives, 
\begin{equation}
\label{temp1}
\begin{aligned}
    &\lv \V[\tF] - \V[F]\rv = 2\e\l( 2\E[f_0g_0f_2g_2]- \E[f_0g_0]\E[f_2g_2]\r) + o(\e).
\end{aligned}
\end{equation}
A sufficient condition for $2\E[f_0g_0f_2g_2]- \E[f_0g_0]\E[f_2g_2]$ to be bounded is that $f_0,g_0,
f_2, g_2$ are all bounded. The condition \eqref{cond_11} guarantees $f_2,g_2$ are bounded, and since
$f_0 = \g \E_a[\nb_\th \Q(\sm, a)] - \Q(\sm, \am)$, $g_0 = r(\sm, \sm, \am) + \g \E_a[\Q(\sm,a)] -
\Q(\sm, \am)$, as long as
\begin{equation*}
\begin{aligned}
     &\lv\E_{a}[\nb \Q(s,a;\th)] - \Q(s,a)\rv\leq C, \qd \forall s\in\S, a\in\A, \th\in\R^\dth;\\
     &\lv\E_{a} [\Q(s,a;\th)]  - \Q(s,a;
    \th)\rv\leq C, \qd \forall s\in\S, a\in\A, \th\in\R^\dth;\\
     &\lv r(s,s,a) \rv \leq C, \qd \forall s\in\S, a\in\A;
\end{aligned}
\end{equation*}
then the coefficient of $\e$ is bounded. This implies that $\lv \V[\tF] - \V[F]\rv = O(\e)$, completing the proof.
\end{proof}

\begin{lemma}
  \label{lemma: control diff uc-gd sc-gd}
  Let $\dsp f(s;\th) = \max_{a'\in\A} \Qst(s,a;\th)$. Suppose that $f(s; \th)$ is continuous in $s\in\S$ and
  $\pt_sf(s;\th)$, $\pt_s^2f(s;\th)$ exist almost surely. Further assume that
  $\dsp\sup_{s\in\S,\th\in\R^\dth}\lv\pt_s f(\sm;\th)\rv$,
  $\sup_{s\in\S,\th\in\R^\dth}\lv\pt_s\nb_\th f(\sm;\th)\rv$, $\dsp\sup_{s\in\S,a\in\A}
  \lv\pt_{s'}r(s,s,a) \rv \leq C $ a.s. Then the difference between the gradients in the US and SC
  algorithms is bounded by
  \begin{equation*}
    \begin{aligned}
      &\lv\E[\tF] - \E[F]\rv \leq 2\s^2C^2\e + o(\e),
    \end{aligned}
  \end{equation*}
  In addition, if $\dsp \lv \max_a Q(s,a;\th) - Q(s,a;\th)\rv,\lv \nb_\th\max_a\Q(s,a;\th) - \nb_\th
  Q(s,a;\th)\rv, \lv r(s,s,a)\rv \leq C$ almost surely in $s\in \S, a\in\A, \th\in \R^\dth$, then
  \begin{equation*}
    \begin{aligned}
      &\lv \V[\tF] - \V[F]\rv \leq O(\e^2).
    \end{aligned}
  \end{equation*}
\end{lemma}

From the above theorem, we see that the magnitude of the difference is related to $\pt_s\Qst$,
$\pt_s\nb_\th\Qst$, and $\pt_{s'}r$. We can control the first two terms through the approximating
function space. This implies that if the reward $r(s',s,a)$ changes slowly w.r.t. $s'$, then the
sample-cloning algorithm for $Q$-control performs better.

\begin{proof}
  The proof of this Lemma is almost the same as the one of Lemma \ref{lemma: diff uc-gd sc-gd}, except
  that $f_i, g_i$ are the ones defined in the proof of Lemma \ref{lemma: diff gd BFF control}, that
  is, in \eqref{temp2}, \eqref{temp3}. Therefore, we omit the proof here.
\end{proof}

\subsubsection{Difference for the whole process} 
The p.d.f. of the parameters during the SC algorithm satisfies the  equation
\begin{align}
  \pt_t\tp = \nb\cdot\l[ \E [\tF] \tp + \frac{\eta}2\nb\cdot\l(\V[\tF]\tp\r)\r].\label{eq: pdf ds}
\end{align}
Therefore, the difference of the p.d.f.s $\td = p - \tp$ satisfies
\begin{align}
  &\pt_t\td = \nb\cdot\l[ \E [F] \td + \frac{\eta}2\nb\cdot\l(\V[F]\td\r)\r] + \nb\cdot\l[ \l(\E[F] - \E[\tF]\r)\tp + \frac{\eta}2\nb\cdot\l(\l( \V[F] - \V[\tF]\r)\tp\r)\r].\label{eq: diff ds}
\end{align}
Using this observation, we can prove the following theorem.

\begin{theorem}
\label{thm: diff of pdf 2}
The difference $\td$ of the p.d.f. between US and SC satisfies,
\begin{equation*}
  \begin{aligned}
    \ll \td(t) \rl_\st  \leq& e^{-\frac{\lam(\b)}4t}\ll p(0) - \pinf \rl_\st + O\l(\e\r) \sqrt{1-e^{-\frac{\lam(\b)}2t}}.
  \end{aligned}
\end{equation*}
\end{theorem}
Unlike the evolution of $\hd$ in Theorem \ref{thm: diff of pdf}, the difference between SC and US
will eventually decay to $O(\e)$ instead of $O(\e\sqrt{\E[\d_\st^2]}\eta^{-\frac{\dthst}{4}})$. As a
result, the error of SC is much larger than that of BFF.

\begin{proof}
The analysis of $\ll \td \rl^2_\st$ is similar to the analysis of $\ll\hd \rl^2_\st $ before
applying Lemma \ref{lemma: d p^2}. Therefore, similar to \eqref{eq: proof hd 1}, we have
\begin{equation*}
    \begin{aligned}
    \frac12\pt_t \ll \td \rl^2_\st \leq&  - \frac\lam4\ll \td \rl^2_\st + O(\e^2)\int \frac{p^2}{\pinf} d\th  + O(\e^2\eta^2)\int \lv\nb\l(\frac{p - \pinf}{\pinf}\r)\rv^2\pinf d\th \\
    \leq & - \frac\lam4\ll \td \rl^2_\st + O(\e^2)\ll d\rl^2_\st + O(\e^2)  + O(\e^2\eta^2)\int \lv\nb\l(\frac{d}{\pinf}\r)\rv^2\pinf d\th ,
    \end{aligned}
\end{equation*}
where $d = p - \pinf$. Combining the above equation with \eqref{eq: proof d 1} and taking $d = p -
\pinf$, we have
\begin{equation*}
  \begin{aligned}
  \frac12\pt_t \l(\ll \hd \rl^2_\st + \ll d \rl^2_\st \r)  \leq&  - \frac\lam4\l(\ll \hd \rl^2_\st + \ll d \rl^2_\st \r)    + O(\e^2) \\
    \pt_t \l[e^{\frac\lam2t}\l(\ll \hd \rl^2_\st + \ll d \rl^2_\st \r)\r] 
    \leq& O(\e^2)e^{\frac\lam2t}.\\
    \end{aligned}
\end{equation*}
Integrating the above inequality on both sides leads to
\begin{equation*}
    \begin{aligned}
    \l(\ll \hd(t) \rl^2_\st + \ll d(t) \rl^2_\st \r)  \leq& e^{-\frac\lam2t}\l(\ll \hd(0) \rl^2_\st + \ll d(0) \rl^2_\st \r) + O\l(\e^2\r) (1-e^{-\frac\lam2t}).
    \end{aligned}
\end{equation*}
Since $\hd(0) = 0$, the above inequality is equivalent to, 
\begin{equation*}
    \begin{aligned}
    \ll \hd(t) \rl^2_\st  \leq& e^{-\frac\lam2t}\ll d(0) \rl^2_\st + O\l(\e^2\r) (1-e^{-\frac\lam2t}).
    \end{aligned}
\end{equation*}
\end{proof}

\subsection{BFF algorithm for $Q$-control}
\label{appendix: control algo}
\begin{algorithm}
\label{algo: control BFF para}
\caption{BFF}
\begin{algorithmic}[1]
    \REQUIRE $\eta$: learning rate
    \REQUIRE $Q^*(s; \th) \in \R^{|\A|}$: nonlinear function approximation of $Q$ parameterized by $\th$
    \REQUIRE $\jc(s_m, a_m, s_{m+1}; \th) = r(s_{m+1}, s_m, a_m) + \g \max_a Q^*(s_{m+1}, a; \th) - Q^*(s_m, a_m; \th)$
    \REQUIRE $\th_0$: Initial parameter vector
    \STATE $m \gets 0$
    \WHILE{$\th_m$ not converged}
        \STATE $s'_{m+1} \gets s_m + (s_{m+2} - s_{m+1})$
        \STATE $\hat{F}_m \gets \jc(s_m, a_m, s_{m+1}; \th_m)\nabla_\th \jc(s_m, a_m, s'_{m+1}; \th_m)$
        \STATE $\th_{m+1} \gets \th_m - \eta \hat{F}_m$
        \STATE $m \gets m+1$
    \ENDWHILE
\end{algorithmic}
\end{algorithm}

\begin{algorithm}
  \caption{BFF (tabular case)}
  \label{algo: bff tab control}
  \begin{algorithmic}[1]
    \REQUIRE $\eta$: Learning rate
    \REQUIRE $Q \in \R^{|\S|\times |\A|}$: matrix of $Q(s,a)$ values
    \REQUIRE $\jc(s_m, a_m, s_{m+1}) = r(s_{m+1}, s_m, a_m) + \g \max_a Q(s_{m+1}, a) - Q^\pi(s_m, a_m)$
    \STATE $m \gets 0$
    \WHILE{$Q^\pi$ not converged}
    \STATE $s'_{m+1} \gets s_m + (s_{m+2} - s_{m+1})$
    \STATE $\hat{F}_m \gets 0 \in \R^{|\S|\times |\A|}$
    \STATE $\hat{F}_m(s_m, a_m) \gets -\jc(s_m, a_m, s_{m+1})$
    \STATE $a_{m+1}^* \gets \argmax_a Q(s'_{m+1}, a)$
    \STATE $\hat{F}_m(s'_{m+1}, a_{m+1}^*) \gets \g\jc(s_m, a_m, s_{m+1})$
    \STATE $Q^\pi \gets Q^\pi - \eta \hat{F}_m$
    \STATE $m \gets m+1$
    \ENDWHILE
  \end{algorithmic}
\end{algorithm}

\subsection{Multiple future steps, tabular case}
\label{appendix: nbff tabular}
Algorithm \ref{algo: bff multiple tab} details the multiple-future-step version of BFF for the tabular control case.
\begin{algorithm}[H]
  \caption{BFF (Multiple-future-step, tabular case)}
  \label{algo: bff multiple tab}
  \begin{algorithmic}[1]
    \REQUIRE $\eta$: Learning rate
    \REQUIRE $Q \in \R^{|\S|\times |\A|}$: matrix of $Q(s,a)$ values
    \REQUIRE $\{\a_k\}_{k=1}^n$
    \REQUIRE $\jc(s_m, a_m, s_{m+1}) = r(s_{m+1}, s_m, a_m) + \g \max_a Q(s_{m+1}, a) - Q^\pi(s_m, a_m)$
    \STATE $m \gets 0$
    \WHILE{$Q^\pi$ not converged}
    \STATE $\hat{F}_m \gets 0 \in \R^{|\S|\times |\A|}$
    \FOR{$k = 1\ldots, n$}
    \STATE $s'_{m+k} \gets s_{m+k-1} + (s_{m+k+1} - s_{m+k})$
    \STATE $\hat{F}_m(s_m, a_m) \gets \hat{F}_m(s_m, a_m) - \jc(s_m, a_m, s_{m+1})$
    \STATE $a_{m+k}^* \gets \argmax_a Q(s'_{m+k}, a)$
    \STATE $\hat{F}_m(s'_{m+k}, a_{m+k}^*) \gets \hat{F}_m(s'_{m+k}, a_{m+k}^*) + \a_k\g\jc(s_m, a_m, s_{m+1})$
    \ENDFOR
    \STATE $Q^\pi \gets Q^\pi - \eta \hat{F}_m$
    \STATE $m \gets m+1$
    \ENDWHILE
  \end{algorithmic}
\end{algorithm}

\subsection{Experiment details} \label{experiment details}
\subsubsection{Tabular evaluation case}
The training procedure is as follows. We generate a long trajectory of length $T=10^7$ from the MDP
dynamics using a fixed policy $\pi(a | s) = \frac12 + a\frac{\sin(s)}{5}$. We use a learning rate of
$\eta = 0.5$ and a batch size of 50 for each of the methods. We find the exact matrix $Q^*$ by first
forming a Monte Carlo estimate of the transition matrix $\mP$ based on 50,000 repetitions per entry,
then forming the expected reward vector $R$ and solving the Bellman equation based on this estimate
for $\mP$.

\subsubsection{Tabular control case}
We find the exact $Q^*$ by running US on a trajectory of length $10^8$ with batch size 1000 and learning rate 0.5 to obtain an approximation $Q^1$. We then refine $Q^1$ by training via US on a trajectory of length $10^7$ with batch size 10000 and a learning rate of 0.1 to obtain the true $Q^*$. We confirm the correctness of $Q^*$ via Monte Carlo (not shown).

We test each of the methods (US, SC, and BFF) on a trajectory of length $5\times 10^7$ with a learning rate of 0.5 and a batch size of 100. The results are shown in Figure \ref{fig: tab control}. BFF outperforms SC by a wide margin and has performance comparable to US. Using a greater number of future steps to approximate the BFF gradient improved its performance marginally.

\begin{figure}[h]
\begin{centering}
  \includegraphics[width=\linewidth]{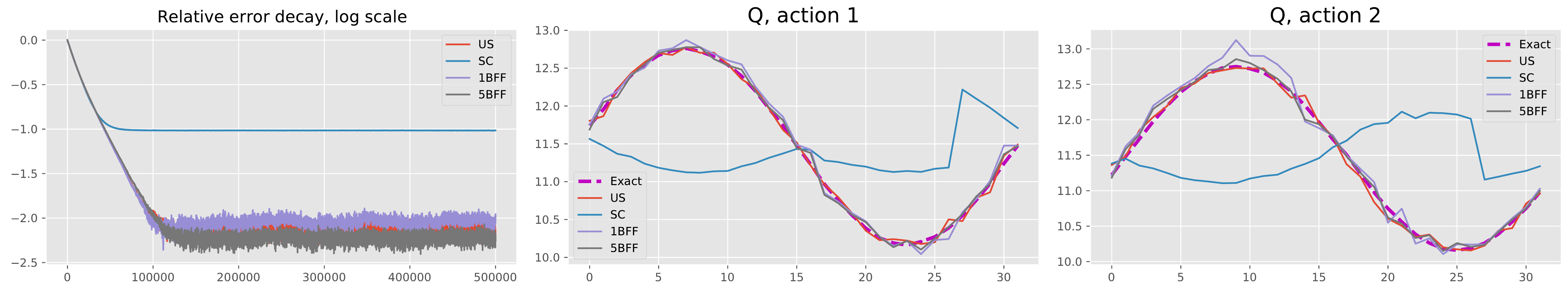}
  \caption{Results of each method for $Q$-control in the tabular case. SC is unable to learn an accurate approximation for $Q$, while BFF's performance is almost indistinguishable from US. Using 5 future steps to compute the BFF approximation helped improve its performance slightly.}
  \label{fig: tab control}
\end{centering}
\end{figure}

\subsubsection{The PD algorithm}
\label{appendix: pd}
The primal dual method transfer the  minimization problem to a mimimax problem, that is, 
\begin{equation*}
    \min_{\th} \max_{\o} \ \E_{(\sm,\am)}[ \d(\sm,\am;\th) y(\sm,\am;\o) - \frac12y(\sm,\am;\o)^2  ]
\end{equation*}
Therefore SGD applied to the above minimax problem does not have the double sampling problem
anymore. The algorithm updates the parameters in the following way,
\begin{align}
    &\o_{k+1} = \o_k + \b(\d(\sm,\am;\th_k))\nb_\o y(\sm,\am;\o_k) - y(\sm,\am;\o_k) \nb_\o y(\sm,\am;\o_k);\nonumber\\
    &\th_{k+1} = \th_k - \eta(\nb_\th\d(\sm,\am;\th_k) y(\sm,\am;\o_{k+1})).\nonumber
\end{align}
We usually set $y(s,a;\o)$ to be the same model as $Q(s,a;\th)$.

\subsubsection{$Q$-evaluation, continuous case}
\label{appendix: cts}
We use a neural network with two hidden layers to approximate Q. Each hidden layer has 50
neurons. The activations are $\cos(x)$ for the hidden layers and identity for the output layer.

The training procedure is similar to the tabular case. We generate a trajectory of length $10^6$ and
run BFF, SC, and US with batch size $M=50$ and learning rate $\eta=0.1$. We also train via PD with
$\beta=\eta=0.1$ and all other hyperparameters identical. We compute the exact $Q$ by running US on
a trajectory of length $10^7$.

As discussed previously, PD has unstable performance. In Figure \ref{fig: eval pd} we plot the
results of the 10 different runs of PD. We compare to the error of BFF and the exact $Q$ function
for reference.

\begin{figure}[h]
\begin{centering}
  \includegraphics[width=\linewidth]{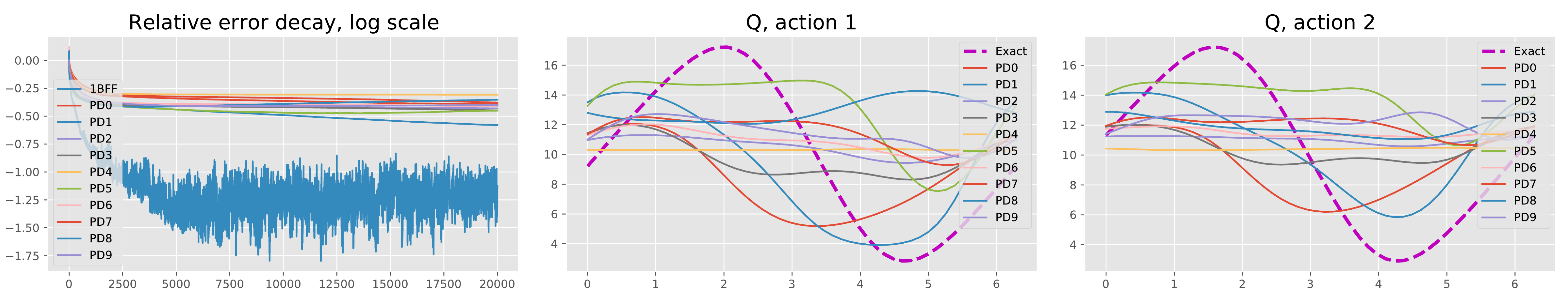}
  \caption{Results for each of the 10 runs of PD for fixed-policy $Q$-evaluation. We include the
    error plot for BFF as well as the exact $Q$ function for reference. The shapes of the $Q$ function
    learned by PD are quite unstable with large variation between runs.}
  \label{fig: eval pd}
\end{centering}
\end{figure}

\subsubsection{$Q$-control, continuous case}
\label{appendix: cts control}
The training procedure is identical to the continuous $Q$-evaluation experiment (with the same
hyperparameters, trajectory length, etc.), but we generate the trajectory with the fixed behavior
policy which samples an action uniformly at random. PD is again unstable and we report the results
of the 10 runs below.
\begin{figure}[h]
  \begin{center}
    \includegraphics[width=\linewidth]{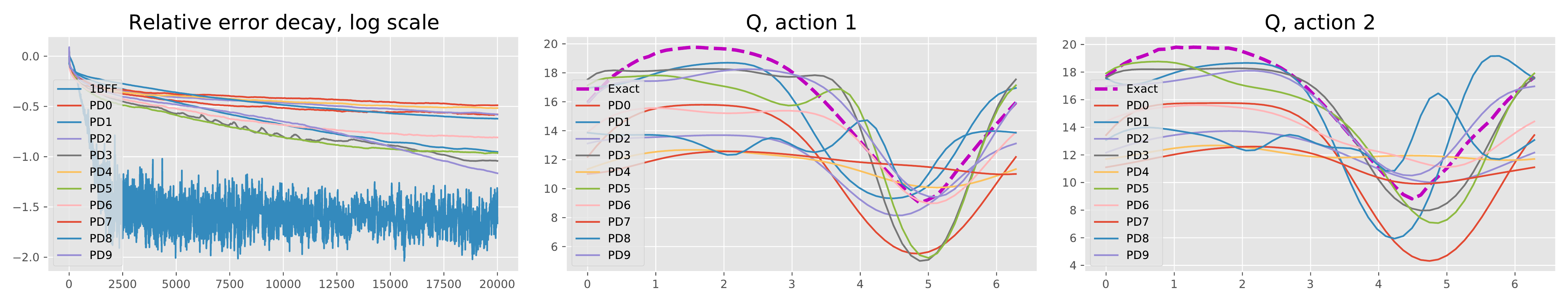}
  \end{center}
  \caption{Results for each of the 10 runs of PD for $Q$-control. We include the error plot for BFF
    as well as the exact $Q$ function for reference. As in the case of $Q$-evaluation, there is
    large variation in the quality of the learned $Q$ function.}
  \label{fig: nn_control_Q}
\end{figure}

\subsubsection{CartPole}
\label{appendix: cartpole}
We approximate $Q$ with a neural network with a single hidden layer of size 100. The hidden layer
has ReLU activations. For both BFF and sample-cloning, we train using Adam with the default settings
for $\beta_1$ and $\beta_2$ ($\beta_1 = 0.9$, $\beta_2 = 0.999$). For all of the methods, we use
batch size 50 and experience replay storing the 10,000 most recent experiences in the training
trajectory. We train for 200 episodes. We also use an $\epsilon$-greedy approach to generate the
trajectory. Initially, we set $\epsilon=1$ (so the agent acts completely randomly at the beginning
of training), and decay $\epsilon$ by $0.99$ after each parameter update. We stop decaying
$\epsilon$ when it reaches $0.1$, so there is always some randomness in our training actions to
prevent getting stuck on an ineffective policy.

For the PD algorithm, We tried fixed values for $\beta$ and $\eta$, as well as decaying $\beta$ and
$\eta$ with different starting values and with the decay recommended in
\cite{wang2017stochastic}. The results in Figure \ref{fig: cartpole} have $\beta_k = 0.1 \times
k^{-3/4}$ and $\eta_k = 0.1 \times k^{-1/2}$, where $\beta_k$ and $\eta_k$ denote the parameters
used for the $k$-th step.



\begin{thebibliography}{18}
\providecommand{\natexlab}[1]{#1}
\providecommand{\url}[1]{\texttt{#1}}
\expandafter\ifx\csname urlstyle\endcsname\relax
  \providecommand{\doi}[1]{doi: #1}\else
  \providecommand{\doi}{doi: \begingroup \urlstyle{rm}\Url}\fi



\bibitem[Baird(1995)]{baird1995}
Leemon Baird.
\newblock Residual algorithms: Reinforcement learning with function approximation.
\newblock \emph{Machine Learning Proceedings}, pg. 30–37, 1995.

\bibitem[Bhatnagar et al.(2009)]{Bhatnagar2018}
 Shalabh Bhatnagar, Doina Precup,  David Silver, Richard S. Sutton, Hamid R. Maei, Csaba Szepesv{\'a}ri.
\newblock Convergent Temporal-Difference Learning with Arbitrary Smooth Function Approximation.
\newblock \emph{NIPS}, 2009.

\bibitem[Bissell et al. (2020)]{bissell2020}
David Bissell, Thomas Birtchnell, Anthony Elliott, and Eric L. Hsu.
\newblock Autonomous automobilities: The social impacts of driverless vehicles. \newblock \emph{Current Sociology}, 2020.

\bibitem[Brockman et~al. (2016)]{gym}
Greg Brockman, Vicki Cheung, Ludwig Pettersson, Jonas Schneider, John Schulman, Jie Tang, and Wojciech Zaremba.
\newblock OpenAI Gym.
\newblock \emph{arXiv preprint arXiv:1606.01540}, 2016.

\bibitem[Carlsson (2012)]{carlsson2012}
Bo Carlsson.
\newblock Technological systems and economic performance: the case of factory automation.
Springer Science \& Business Media, 2012.

\bibitem[Dai et al.(2018)]{Dai2018}
Bo Dai, Albert Shaw, Lihong Li, Lin Xiao, Niao He, Zhen Liu, Jianshu Chen, Le Song.
\newblock SBEED: Convergent reinforcement learning
with nonlinear function approximation. In International.
\newblock \emph{ICML}, 2018.

\bibitem[Gu et al. (2017)]{gu2017}
Shixiang Gu, Ethan Holly, Timothy Lillicrap, and Sergey Levine.
\newblock Deep Reinforcement Learning for Robotic Manipulation with Asynchronous Off-Policy Updates.
\newblock \emph{ICRA}, 2017.

\bibitem[van Hasselt et~al. (2015)]{hasselt2015}
Hado van Hasselt, Arthur Guez, and David Silver.
\newblock Deep Reinforcement Learning with Double $Q$-learning
\newblock In \emph{AAAI
Conference on Artificial Intelligence}, AAAI, 2016.

\bibitem[Hu et~al.(2017)]{hu2017diffusion}
Wenqing Hu, Chris~Junchi Li, Lei Li, and Jian-Guo Liu.
\newblock On the diffusion approximation of nonconvex stochastic gradient
  descent.
\newblock \emph{arXiv preprint arXiv:1705.07562}, 2017.

\bibitem[Kingma et~al. (2014)]{adam}
Diederik P. Kingma and Jimmy Ba.
\newblock Adam: A Method for Stochastic Optimization.
\newblock In \emph{3rd International Conference on Learning Representations}, ICLR, 2015.

\bibitem[Leontief et al. (1986)]{leontief1986}
Wassily Leontief and Duchin Faye.
\newblock The Future Impact of Automation on Workers.
\newblock Oxford University Press, 1986.

\bibitem[Li et~al.(2017)]{li2017stochastic}
Qianxiao Li, Cheng Tai, et~al.
\newblock Stochastic modified equations and adaptive stochastic gradient
  algorithms.
\newblock In \emph{Proceedings of the 34th International Conference on Machine
  Learning-Volume 70}, pages 2101--2110. JMLR. org, 2017.

\bibitem[Liu et al.(2015)]{Liu2015}
Bo Liu, Ji Liu, Mohammad Ghavamzadeh, Sridhar Mahadevan and Marek Petrik.
\newblock Finite-sample analysis of proximal gradient TD algorithms.
\newblock \emph{UAI}, 2015.

\bibitem[Mahadevan et al.(2011)]{Mahadevan2011}
Sridhar Mahadevan, Bo Liu, Philip Thomas, Will Dabney, Steve Giguere, Nicholas Jacek, Ian Gemp, Ji Liu.
\newblock Proximal reinforcement learning: A new theory of sequential decision making in primal-dual spaces.
\newblock \emph{arXiv preprint}, arXiv:1405.6757, 2014.

\bibitem[Mnih. et al.(2016)]{Mnih2016}
Volodymyr Mnih, Adrià Puigdomènech Badia, Mehdi Mirza, Alex Graves, Timothy P. Lillicrap, Tim Harley, David Silver and Koray Kavukcuoglu.
\newblock Asynchronous methods for deep reinforcement learning.
\newblock \emph{ICML}, pp. 1928–1937, 2016.

\bibitem[Mnih. et al.(2013)]{mnih2013playing}
Volodymyr Mnih, Koray Kavukcuoglu, David Silver, Alex Graves, Ioannis Antonoglou, Daan Wierstra and Martin Riedmiller.
\newblock Playing atari with deep reinforcement learning.
\newblock \emph{arXiv preprint},
arXiv:1312.5602, 2013.

\bibitem[Mnih. et al.(2015)]{Mnih2015}
Volodymyr Mnih, Koray Kavukcuoglu, David Silver, Andrei A. Rusu, Joel Veness, Marc G. Bellemare, Alex Graves, Martin Riedmiller, Andreas K. Fidjeland, Georg Ostrovski, Stig Petersen, Charles Beattie, Amir Sadik, Ioannis Antonoglou, Helen King, Dharshan Kumaran, Daan Wierstra, Shane Legg and Demis Hassabis.
\newblock Human-level control through deep reinforcement learning.
\newblock \emph{Nature}, 518:529–533, 2015.

\bibitem[Rummery \& Niranjan(1994)]{Rummery1994}
Gavin A. Rummery and Mahesan Niranjan.
\newblock \emph{On-line $Q$-learning using connectionist systems}.
\newblock University of Cambridge, Department of Engineering Cambridge, UK, 1994

\bibitem[Sallab et al. (2017)]{sallab2017}
Ahmad El Sallab, Mohammed Abdou, Etienne Perot, and Senthil Yogamani.
\newblock Deep Reinforcement Learning framework for Autonomous Driving.
\newblock \emph{Electronic Imaging}, 2017.

\bibitem[Silver et al.(2016)]{Silver2016}
David Silver, Aja Huang, Chris J Maddison, Arthur Guez, Laurent Sifre, George Van Den Driessche, Julian Schrittwieser, Ioannis Antonoglou, Veda Panneershelvam, Marc Lanctot, et al. 
\newblock Mastering the game of go with deep neural networks and tree search. 
\newblock \emph{Nature}, 529(7587): 484, 2016.

\bibitem[Silver et al.(2017)]{Silver2017}
David Silver, Julian Schrittwieser, Karen Simonyan, Ioannis Antonoglou, Aja Huang, Arthur Guez, Thomas Hubert, Lucas Baker, Matthew Lai, Adrian Bolton, et al.
\newblock Mastering the game of go without human knowledge.
\newblock \emph{Nature}, 550(7676):354, 2017.

\bibitem[Sutton(1988)]{sutton1988learning}
Richard~S Sutton.
\newblock Learning to predict by the methods of temporal differences.
\newblock \emph{Machine learning}, 3\penalty0 (1):\penalty0 9--44, 1988.

\bibitem[Sutton \& Barto(2018)]{sutton2018reinforcement}
Richard~S Sutton and Andrew~G Barto.
\newblock \emph{Reinforcement learning: An introduction}.
\newblock MIT press, 2018. 

\bibitem[Sutton et al.(2009)]{Sutton2009}
Richard S. Sutton, Hamid Reza Maei, Doina Precup, Shalabh Bhatnagar, David Silver, Csaba Szepesv\`ari, and Eric Wiewiora.
\newblock Fast gradient-descent methods for temporal-difference learning with linear function approximation.
\newblock \emph{ICML}, pp. 993–1000, 2009.

\bibitem[Sutton(2008)]{suttongtd2008}
Richard S. Sutton, Csaba Szepesvári and Hamid Reza Maei.
\newblock A convergent O (n) algorithm for off-policy temporal-difference learning with linear function approximation.
\newblock \emph{NIPS}, pg. 1609-1616, 2008.

\bibitem[Wang et~al.(2017)]{wang2017stochastic}
Mengdi Wang, Ethan~X Fang, and Han Liu.
\newblock Stochastic compositional gradient descent: algorithms for minimizing
  compositions of expected-value functions.
\newblock \emph{Mathematical Programming}, 161\penalty0 (1-2):\penalty0
  419--449, 2017.
  
\bibitem[Wang et~al.(2016)]{wang2016accelerating}
Mengdi Wang, Ji~Liu, and Ethan Fang.
\newblock Accelerating stochastic composition optimization.
\newblock In \emph{Advances in Neural Information Processing Systems}, pg.
  1714--1722, 2016.

\bibitem[Watkins(1989)]{Watkin1989}
Christopher J.C.H. Watkins.
\newblock Learning from Delayed Rewards.
\newblock \emph{PhD thesis}, King’s College, University of Cambridge, UK, 1989. 

\bibitem[Zhu et al.(2020)]{zhu2020}
Yuhua Zhu and Lexing Ying
\newblock Borrowing From the Future: An attempt to address double sampling.
\newblock \emph{MSML}, accepted, 2020.





















\end{thebibliography}
\end{document}